\documentclass[reqno]{amsart}%,draft
\usepackage{xcolor}
\usepackage{amsmath}
\usepackage{arydshln}
\usepackage{appendix}%
\allowdisplaybreaks[3]
\numberwithin{equation}{section}%\documentclass[reqno,12pt]{amsart}

\newtheorem{thm}{\indent Theorem}[section]
\newtheorem{cor}[thm]{\indent Corollary}
\newtheorem{lem}[thm]{\indent Lemma}
\newtheorem{prop}[thm]{\indent Proposition}

\newtheorem{rmk}{{\indent\bf Remark}}[section]

\newcommand{\lra}{\longrightarrow}

\newcommand{\ol}{\overline}

\newcommand{\td}{\tilde}

\newcommand{\fr}{\frac}

\newcommand{\edd}{\end{document}}
\newcommand{\be}{\begin{equation}}
\newcommand{\ee}{\end{equation}}

\newcommand{\lagl}{\langle}
\newcommand{\ragl}{\rangle}
\newcommand{\lmx}{\left(\begin{matrix}}
\newcommand{\rmx}{\end{matrix}\right)}
\newcommand{\ldt}{\left|\begin{matrix}}
\newcommand{\rdt}{\end{matrix}\right|}

\newcommand{\tr}{{\rm tr\,}}
\newcommand{\vfi}{\varphi}

\newcommand{\veps}{\varepsilon}

\newcommand{\bbr}{{\mathbb R}}

\newcommand{\bbs}{{\mathbb S}}

\newcommand{\ba}{\begin{array}}
\newcommand{\ea}{\end{array}}
\newcommand{\nnm}{\nonumber}
\newcommand{\beal}{\begin{align}}
\newcommand{\eal}{\end{align}}
\newcommand{\bea}{\begin{eqnarray}}
\newcommand{\eea}{\end{eqnarray}}

\newcommand{\id}{{\rm id}}

\newcommand{\pp}[2]{\fr{\partial #1}{\partial #2}}

\newcommand{\emf}{e^{\fr1m|F|^2}}
\newcommand{\eamf}{e^{\fr am|F|^2}}
\newcommand{\ft}{e^{\fr1m|F|^2}(H+F)}
\newcommand{\fta}{e^{\fr am|F|^2}(cH+bF)}

\begin{document}

\title[The mean curvature flow in the Gaussian space]{On the mean curvature flow of submanifolds\\
in the standard Gaussian space $^\dag$}

\author[A.-M. Li]{An-Min Li}

\author[X. X. Li]{Xingxiao Li}

\author[D. Zhang]{Di Zhang $^*$}

\dedicatory{}

%%%%%%%%%%%%%%% footnote %%%%%%%%%%%%%%%%
%In case \subjclass[2000] command is not effective
%(or the version of amsart.cls is old), write as follows instead:
%\renewcommand{\thefootnote}{\fnsymbol{footnote}}
%\footnote[0]{2000\textit{ Mathematics Subject Classification}.
%Primary 00; Secondary 00.}
%
\thanks{$^\dag$ Research supported by
National Natural Science Foundation of China (No. 11671121, No. 11871197 and No 11971153).}
\thanks{$^*$ The corresponding author}%%%%%%%%%%%% Authors' addresses %%%%%%%%%%%%%

\address{An-Min Li\endgraf
Department of Mathematics\endgraf
Sichuan University\endgraf
Chengdu 610065, Sichuan\endgraf
P.R. China}
\email{anminliscu@126.com}

\address{Xingxiao Li\endgraf
School of Mathematics and Information Sciences\endgraf
Henan Normal University\endgraf
Xinxiang 453007, Henan\endgraf
P.R. China}
\email{xxl@henannu.edu.cn}

\address{Di Zhang\endgraf
Department of Mathematics\endgraf
Sichuan University\endgraf
Chengdu 610065, Sichuan\endgraf
P.R. China}
\email{zhangdi5727@163.com}

\begin{abstract}
In this paper, we study the regular geometric behavior of the mean curvature flow (MCF) of submanifolds in the standard Gaussian metric space $(\bbr^{m+p},e^{-|x|^2/m}\ol g)$ where $(\bbr^{m+p},\ol g)$ is the standard Euclidean space and $x\in\bbr^{m+p}$ denotes the position vector. Note that, as a special Riemannian manifold, $(\bbr^{m+p},e^{-|x|^2/m}\ol g)$ has an unbounded curvature. Up to a family of diffeomorphisms on $M^m$, the mean curvature flow we considered here turns out to be equivalent to a special variation of the ``{\em conformal mean curvature flow}\,'' which we have introduced previously. The main theorem of this paper indicates, geometrically, that any immersed compact submanifold in the standard Gaussian space, with the square norm of the position vector being not equal to $m$, will blow up at a finite time under the mean curvature flow, in the sense that either the position or the curvature blows up to infinity; Moreover, by this main theorem, the interval $[0,T)$ of time in which the flowing submanifolds keep regular has some certain optimal upper bound, and it can reach the bound if and only if the initial submanifold either shrinks to the origin or expands uniformly to infinity under the flow. Besides the main theorem, we also obtain some other interesting conclusions which not only play their key roles in proving the main theorem but also characterize in part the geometric behavior of the flow, being of independent significance.
\end{abstract}

\maketitle

\tableofcontents

\section{Introduction}

Let $M$ be a compact manifold of dimension $m\geq2$, and $(\bbr^{m+p},\ol g)$ ($p\geq1$) the Euclidean space with the standard flat metric $\ol g$.
Then the Riemannian manifold $(\bbr^{m+p},e^{-|x|^2/m}\ol g)$, where $|x|$ denotes the standard norm of the position vector $x\in\bbr^{m+p}$, is exactly the standard Gaussian metric space. Thus, for any immersion $x:M\to\bbr^{m+p}$, we have two induced metrics $g:=x^*\ol g$, $\td g=e^{-\frac1m|x|^2}g$ and, accordingly, we also have two mean curvature vectors $H$, $\td H$. A simple computation shows that these two mean curvature vectors are related to each other by
\be\label{tdH}\td H=e^{\frac1m|x|^2}(H+x^\bot).\ee

As is well known, the Gaussian space $(\bbr^{m+p},e^{-|x|^2/m}\ol g)$ plays important roles both in physics and in mathematics; In particular, as a special Riemannian manifold, it has a typical geometric structure of non-constant curvature. Moreover, while the Gaussian metric $e^{-|x|^2/m}\ol g$ is highly symmetric around the origin, the sectional curvature of $(\bbr^{m+p},e^{-|x|^2/m}\ol g)$ is unfortunately unbounded. In fact, under the canonical frame $e_A=\pp{}{x^A}$ $(A=1,2,...,m+p)$ of $\bbr^{m+p}$, if we denote by $\tilde R_{ABCD}$ $(A,B,C,D=1,2,...,m+p)$ the components of the Riemannian curvature tensor of the Gaussian metric $e^{-|x|^2/m}\ol g$, then the sectional curvature $\tilde K(e_A,e_B)$ of the Gaussian space $(\bbr^{m+p},e^{-|x|^2/m}\ol g)$ along the two-dimensional section $[e_A\wedge e_B]$ at any point $x=(x^1,x^2,\cdots x^{m+p})\in\bbr^{m+p}$ with any pair $A,B$, $A\neq B$, can be directly computed as follows:
\begin{align*}
\tilde K(e_A,e_B)=&-\frac{\tilde R_{ABAB}}{\|e_A\wedge e_B\|^2}\\
=&\fr1me^{\frac1m|x|^2}\Big(2+\fr1m(x^A)^2+\fr1m(x^B)^2-\frac1m|x|^2\Big)\\
=&\frac1me^{\frac1m|x|^2}\Big(2-\fr1m\sum_{C\neq A,B}(x^C)^2\Big)\lra-\infty,
\end{align*}
as $\sum_{C\neq A,B}(x^C)^2\lra+\infty$, where $\|e_A\wedge e_B\|$ is the norm of the exterior $2$-vector $e_A\wedge e_B$ induced by the Gaussian metric $e^{-|x|^2/m}\ol g$.

Motivated by the major development of the classical mean curvature flow, we are very interested in and shall initiate the study of the mean curvature flow (MCF) of submanifolds in the Gaussian space $(\bbr^{m+p},e^{-|x|^2/m}\ol g)$. Due to some of the idea proposed in our earlier study on the conformal MCF (see the statement in the introduction of \cite{l-z}), instead of using the original mean curvature $\td H$ directly, we would like to write analytically the MCF in $(\bbr^{m+p},e^{-|x|^2/m}\ol g)$, using the ``{\em standard}\,'' mean curvature $H$, into the following form of initial problem of a degenerate parabolic partial differential equation:
\be\label{flow0}
\begin{cases}
\pp{F}{t}=\td H\equiv\emf(H+F^\bot),\\
F(\cdot,0)\equiv F_0(\cdot),
\end{cases}
\ee
where $|F|^2\equiv\ol g(F,F)$ is the square norm of $F$ with respect to the standard flat metric $\ol g$.

On the other hand, we all know that (\cite{w}, see also \cite{e} for $p=1$) a tangential velocity in any flow of submanifolds does not change the geometric behavior of the flow, for example, the image $F_t(M^m)$ at each moment $t\in[0,T)$. Therefore, the original MCF \eqref{flow0} can be made geometrically equivalent to the following curvature flow:
\be\label{flow}
\begin{cases}
\pp{F}{t}=\ft,\\
F(\cdot,0)\equiv F_0(\cdot).
\end{cases}
\ee
More generally, for later use in the major argument, we need to consider the following problem of ``{\em modified MCF}\,":
\be\label{flow'}
\begin{cases}
\pp{F}{t}=e^{\frac1ma(t)|F|^2}\big(c(t)H+b(t)F\big),\\
F(\cdot,0)\equiv F_0(\cdot),
\end{cases}
\ee
where $a,b,c\in C^\infty[0,T)$, $a>0$, $b>0$, $c>0$. The most important case is that $a,b,c$ are all given constants. By the way, we shall call a smooth solution $F:M^m\times[0,T)\to\bbr^{m+p}$ to \eqref{flow0} or \eqref{flow} or \eqref{flow'} {\em regular} if for each $t\in[0,T)$, the corresponding map $F_t:M^m\to\bbr^{m+p}$ is an immersion, where $F_t(\cdot)=F(\cdot,t)$; Such a regular solution $F$ is called {\em maximal} if the time interval $[0,T)$ is maximal.

To give our present research a sound background, a brief review here of the well studied classical MCF of submanifolds, especially in the real space forms, seems rather necessary.

As indicated in the literature, in order to describe the formation of grain boundaries in annealing metals, the MCF was first studied in 1956 by Mullins (\cite{mull}). Later, Brakke (\cite{bra}) introduced the motion of a submanifold in arbitrary codimension driven by the mean curvature vector, and constructed a generalized varifold solution for all time. Since then there have been many interesting results on the MCF, especially for hypersurfaces in Euclidean space, the most important case. To cite a few, Huisken (\cite{hui84}) showed that, under the MCF, any compact and uniformly convex hypersurface in the Euclidean space is convergent to a ``round point" in finite time. More generally in the case of higher codimension, Andrews and Baker proved (\cite{a-b} or \cite{b}) that, for a compact initial submanifold of the MCF in the Euclidean space, if its second fundamental form and mean curvature satisfy a suitable pinching condition, then the corresponding moving submanifold must be convergent to a ``round point" in finite time. This last theorem was later generalized to MCFs in other two non-flat space forms, see Baker (\cite{b}) and Liu-Xu-Ye-Zhao (\cite{l-x-y-z11}, \cite{l-x-y-z}). It should be remarked that, in order to extend the result in \cite{hui86} by Huisken for the MCF of hypersurfaces in any more general Riemannian manifold $(N,\ol g)$ with bounded geometry, and the result in \cite{a-b} by Andrews and Baker for the MCF in Euclidean space, Liu, Xu and Zhao considered in \cite{l-x-z12} the MCF of submanifolds with higher codimension in the geometrically bounded Riemannian manifold $(N,\ol g)$ and proved a general convergence theorem. As for other interesting progresses on the MCFs, we refer the readers to, for example, the references \cite{andr}, \cite{b-n}, \cite{l-x}, \cite{p-s-x}, \cite{smo11} and \cite{w} etc.

Now come back to the curvature flow \eqref{flow}. We remark that the flow \eqref{flow} turns out to be another special interesting case of the so-called ``{\em modified mean curvature flow with an external force}\,'' (see (1.1) in \cite{l-z}). Besides the standard MCF and \eqref{flow}, this modified mean curvature flow with an external force also has some other important special cases that have been extensively studied by many authors. For example, the MCF with density (see \cite{b-r10} and \cite{b-r14} for Gaussian MCF in real space forms);  that in the Euclidean space $\bbr^n$ with the external force in the direction of position vector (\cite{g-l-w}, \cite{s-s}); and some more general flows in $\bbr^n$ when the external force is taken to be the gradient of certain smooth functions $\psi\in C^\infty(\bbr^n)$ (\cite{l-s}, \cite{l-j07} and \cite{l-j08}), which are related to the study of the Ginzburg-Landau vortex (\cite{j-x03} and \cite{j-l06}), or when the external force is chosen to be a closed conformal vector field (\cite{a-l-r} and \cite{mntl}).

As the first article in our study on the MCF of submanifolds in the Gaussian space, the present paper will focus on the general evolution trend of the moving submanifold within the regular time interval. In other words, we shall make some relevant characterizations on how the first singularity of the possible solution of \eqref{flow} comes.

Now our main theorem in the present paper can be stated as follows:

\begin{thm}[The main theorem]\label{main1}
Let $F_0:M^m\to\bbr^{m+p}$ be an immersed compact submanifold satisfying $|F_0|^2\neq m$ everywhere on $M^m$. Denote
$$
\hat T_1:=\fr m{2(m-\max|F_0|^2)}\left(1-e^{-\frac1m\max|F_0|^2}\right),\quad \hat T_2:=\fr m{2(\min|F_0|^2-m)}e^{-\frac1m\min|F_0|^2}.
$$
Then there exists a maximal regular solution $F:M^m\times[0,T)\to\bbr^{m+p}$ to the flow \eqref{flow}. Furthermore,

(1) If $\min|F_0|^2<m$, then we have either
$$T<\hat T_1 \text{\ \ and\ \ }\lim\limits_{t\to T}\max|h|^2=+\infty$$
or
$$T=\hat T_1 \text{\ \ and\ \ }\lim\limits_{t\to\hat T_1}\max|F|^2=0,$$
namely, $F_t(M^m)$ converges to the origin as $t$ tends to $\hat T_1$;

(2) If $\max|F_0|^2>m$, then we have either $T<\hat T_2$ and one of the following two holds:

(i) $\lim\limits_{t\to T}\max|h|^2=+\infty$;

(ii) $\lim\limits_{t\to T}\max|F|^2=+\infty$;
\\
or $T=\hat T_2$ and $\lim\limits_{t\to\hat T_2}\min|F|^2=+\infty$.
\end{thm}

We shall also prove some other characterizations of the flow \eqref{flow}. For example, a blow-up theorem as

\begin{thm}[see Theorem \ref{thm5.1}]\label{main2}
Let $F:M^m\times[0,T)\to\bbr^{m+p}$ be a maximal regular solution of the curvature flow \eqref{flow} such that $|F_0|^2\neq m$ everywhere. Then $T<+\infty$ and either
$$\lim\limits_{t\to T}\max|F|^2=+\infty,\quad\text{or}\quad \lim\limits_{t\to T}\max|h|^2=+\infty,$$
\end{thm}\noindent
and a theorem characterizing how the spherical submanifolds evolve under the flow \eqref{flow'}:

\begin{thm}[see Theorem \ref{thm4.5}]\label{main3}
Let $F:M^m\times[0,T)\to\bbr^{m+p}$ be a maximal regular solution of the flow \eqref{flow'} with both $a$ and $b$ being constant. If the initial submanifold $F_0:M^m\to\bbr^{m+p}$ is contained in a standard hypersphere $\bbs^{m+p-1}(R_0)$ of radius $R_0$ which is centered at the origin, i.e., $|F_0|\equiv R_0$, then $F_t(M^m)$ is always kept on a likewise standard hypersphere. Furthermore,

(1) If $(b|F|^2-mc)(0)<0$ and $c'\geq 0$, then
\be\label{t1}
T\leq T_1:=
\fr m{2ab\left(\fr{c(0)}bm-R^2_0\right)}\left(1-e^{-\frac{a}mR_0^2}\right).
\ee
Furthermore, in the case that $T=T_1$, $F_t(M^m)$ will be convergent to the origin as $t\to T_1$;

(2) If $(b|F|^2-mc)(0)>0$ and $c'\leq 0$, then
\be\label{t2}
T\leq T_2:=\fr m{2 ab\left(R^2_0-\fr{c(0)}bm\right)}e^{-\frac{a}mR_0^2}.
\ee
Furthermore, in the case that $T=T_2$, it must hold that $\lim\limits_{t\to T_2}|F_t|^2=+\infty$;

(3) If  $(b|F|^2-mc)(0)=0$ and $c$ is constant, then $T=+\infty$ and $F_t(M^m)$ is always kept on the fixed standard hypersphere $\bbs^{m+p-1}(R_0)$ for all $t\geq 0$.
\end{thm}

Theorem \ref{main3} apparently has the following corollary which characterizes how a standard hypersphere evolves.

\begin{cor}[see Corollary \ref{cor4.6}]
Let $F:M^{n-1}\times[0,T)\to\bbr^n$ be a maximal regular solution to the flow \eqref{flow'} with $a,b,c$ all being constant. Suppose that the initial submanifold $F_0:M^{n-1}\to\bbr^n$ of \eqref{flow'} is a standard hypersphere centered at the origin, then $F_t:M^{n-1}\to\bbr^n$ remains a likewise standard hypersphere for each $t\in [0,T)$. Moreover,

(1) If $b|F_0|^2<(n-1)c$ and $T_1$ is given by \eqref{t1} with $R_0=|F_0|$ and $m=n-1$, then $T=T_1$ and $\lim\limits_{t\to T_1}|F_t|^2=0$;

(2) If $b|F_0|^2>(n-1)c$ and $T_2$ is given by \eqref{t2} with $R_0=|F_0|$ and $m=n-1$, then $T=T_2$ and $\lim\limits_{t\to T_2}|F_t|^2=+\infty$;

(3) If $b|F_0|^2=(n-1)c$, then $T=+\infty$ and $F_t\equiv F_0$ for all $t\geq 0$.
\end{cor}

\begin{rmk}\rm
For the MCF in the Gaussian space $(\bbr^{m+p},e^{-|x|^2/m}\ol g)$, the most natural and interesting problem for us to study, could be the singular behavior of the limit submanifold of the solution to the MCF \eqref{flow}, on which either the curvature or the position would blow up as is concluded in Theorem \ref{main2}. This problem seems a little more subtle but surely attracting. We reasonably expect that, under some necessary geometric or analytic conditions on the initial submanifold, the regular submanifolds $F_t$ ($t\in[0,T)$) given in the flow, possibly after some suitable rescaling and/or a suitable reparameterization of time, will have a geometrically simple limit, as is seen in the most theorems obtained for the MCF of submanifolds in the real space forms, particularly in the standard Euclidean space.
\end{rmk}

To end the introduction, we shall brief the organization of this paper. In Section 2, by using the well-known trick of DeTurck, we reiterate the process similar to what we did in \cite{l-z}, showing the short-time existence and uniqueness of the solution to the curvature flow \eqref{flow'} (see Theorem \ref{exiuni} in section \ref{s2}). Section 3 contains only some direct computations that give us the necessary evolution equations for a number of basic quantities. Section 4 starts the main part of the paper, providing a few key conclusions that characterize in part the properties of the concerned flow, which will also be used in the proof of the main theorem. Section 5 mainly concerns the key estimates together with the argument to reach a blow-up theorem. In Section 6, the final one, we shall combine all discussions done earlier to complete the proof of the main theorem (see Theorem \ref{thm6.1}).

\section{The short-time existence and the uniqueness of the solution}\label{s2}

Same as the standard MCF we have in the literature, Equation \eqref{flow'} is generally a degenerate parabolic partial differential equation. By using the
``{\em DeTurck trick}\,'' (see, for example, \cite{b}, \cite{dtu} etc), we are to prove in this section the short-time existence and uniqueness of solution of \eqref{flow'} for each initial immersion $F_0:M^m\to\bbr^{m+p}$.

First we need to make some preparations for notations. For a given local coordinate system $(u^i)$ on $M$ and a smooth real-valued or vector-valued function $f$, write
$$e_i:=\pp{}{u^i},\quad\omega^i:=du^i,\quad  f_i:=\pp{}{u^i}(f)\equiv\pp{f}{u^i}.$$
Moreover, as done in \cite{l-z}, we always denote by $\nabla^g$ the Levi-Civita connection of any given Riemannian metric $g$ on $M^m$. Now arbitrarily fix a metric $\mathring g$ and then define a metric-dependent vector field $W\equiv W(g)$ on $M$ by $W(g):=\tr_g(\nabla^g-\nabla^{\mathring g})$. In particular, if $F:M^m\to\bbr^{m+p}$ is an immersion, then we get an immersion-dependent vector field $W(F)\equiv W(g_F)$ where $g_F$ is the induced metric via $F$ of the standard metric $\ol g$ on $\bbr^{m+p}$.

Introduce the following {\em DeTurck MCF} for the flow \eqref{flow'}:
\be\label{dtmcf}
\pp{\hat F}{t}=e^{\fr am|\hat F|^2}(cH_{\hat F}+c\hat F_*W(\hat F)+b\hat F),
\ee
where $H_{\hat F}$ is the mean curvature of the time-dependent immersion $\hat F_t:M^m\to(\bbr^{m+p},\ol g)$. For the given coordinate system $(u^i)$ on $M$, it holds that
$$
F_{t*}(e_i)=e_i(F)\equiv F_i,\quad \hat F_{t*}(e_i)=e_i(\hat F)\equiv\hat F_i,
$$
where $F$ and $\hat F$ on the right hand side are viewed as $\bbr^{m+p}$-valued functions. Denote respectively by $\Gamma^k_{ij}$ and $\mathring\Gamma^k_{ij}$ the Christoffel symbols for $g$ and $\mathring g$ and write
$$
W(g)=W^ke_k=g^{ij}(\Gamma^k_{ij}-\mathring\Gamma^k_{ij})e_k.
$$
Then the two flows \eqref{flow'} and \eqref{dtmcf} have the following local representations respectively:
\be\label{cfmcf1}
\pp{F^A}{t}=\eamf\left(cg_F^{ij}F^A_{,ij}+bF^A\right)
\ee
\be\label{dtmcf1}
\pp{\hat F^A}{t}=e^{\fr am|\hat F|^2}\left(cg_{\hat F}^{ij}\hat F^A_{,ij}+c\hat F^A_kW^k(\hat F)+b\hat F^A\right)\equiv e^{\fr am|\hat F|^2}\left(cg_{\hat F}^{ij}\hat F^A_{;ij}+b\hat F^A\right),
\ee
where the subscript ``$,$\,'' denotes the covariant derivatives w.r.t. the induced metric $g_F$ or $g_{\hat F}$ accordingly, while the subscript $``;$\,'' denotes those w.r.t. the fixed metric $\mathring g$. From \eqref{dtmcf1} it is clearly seen that \eqref{dtmcf} is a (nondegenerate) parabolic equation and thus has a short-time existing solution $\hat F=\hat F(u,t)$, $(u,t)\in M^m\times[0,T)$, according to the standard theory of parabolic equations. So we have a well-defined time-dependent vector field $W=W(\hat F)$ on $M$ for all $t\in[0,T)$.

Now we recall a known existence result as follows:

\begin{lem}[see for example \cite{c-k}, p.82, Lemma 3.15]\label{tmdptfl} If $\{X_t:0\leq t<T\leq\infty\}$ is a continuous time-dependent
family of vector fields on a compact manifold $M$, then there exists uniquely a one-parameter
family of diffeomorphisms
$$\{\vfi_t:M\to M; 0\leq t<T\leq\infty\}$$
defined by $\vfi:M\times[0,T)\to M$
on the same time interval such that
$$\pp{\vfi_t}{t}(u)\equiv\vfi_*\left(\pp{}{t}\right)=X_t(\vfi_t(u)),\quad
\vfi_0(u)=u.
$$
for all $u\in M$ and $t\in[0, T)$.
\end{lem}

Substituting $X$ with $-ce^{\fr am|\hat F|^2}W(\hat F)$ we obtain a family of diffeomorphisms $\vfi_t\equiv\vfi(\cdot,t)$, $0\leq t<T$, on $M$. Write $\td u^i_t=\vfi^i_t(u)$ ($0\leq t<T$). Then $(\td u^i_t)$ is a time-dependent family of local coordinate systems with the parameter $t\in[0,T)$. Define a family of immersions $F_t(\cdot):=\hat F(\vfi_t(\cdot),t)$, $t\in[0,T)$. Then it is easy to see that $\vfi^*_tg_{\hat F_t}=g_{F_t}$.

Given a Riemannian metric $g$ on $M$ and an immersion $F:M^m\to\bbr^{m+p}$, we always use $\lagl\cdot,\cdot\ragl_g$ to denote the induced inner product on the vector bundle $T^r_s(M)\otimes F^*TN$ by the metrics $g$ and $\ol g$, where $T^r_s(M)$ is the $(r,s)$-tensor bundle on $M$. In particular, we shall omit the subscript $g$ in $\lagl\cdot,\cdot\ragl_g$ when $g$ is the induced metric by $F$. From this we compute for any $(u,t)\in M^m\times[0,T)$,
\begin{align*}
\pp{F}{t}(u,t)\equiv& F_*\left(\pp{}{t}\right)(u,t)=\hat F_*\left(\pp{}{t}\right)(\vfi(u,t),t)+(\hat F_t)_*\circ\vfi_*\left(\pp{}{t}\right)(u,t)\\
=&e^{\fr am|\hat F(\vfi(u,t),t)|^2}\left(cg_{\hat F} ^{ij}\hat F_{;ij}+b\hat F\right)(\vfi(u,t),t) -ce^{\fr am|\hat F(\vfi(u,t),t)|^2}(\hat F_t)_*(W(\hat F(\vfi(u,t),t)))\\
=&e^{\fr am|\hat F(\vfi(u,t),t)|^2}\left(cg_{\hat F} ^{ij}(\hat F_{;ij}-\hat F_k(\hat\Gamma^k_{ij}-\mathring\Gamma^k_{ij}))+b\hat F\right)(\vfi(u,t),t)\\
=&e^{\fr am|\hat F(\vfi(u,t),t)|^2}\left(cg_{\hat F} ^{ij} \hat F_{i,j}+b\hat F\right)(\vfi(u,t),t)\\
=&e^{\fr am|F(u,t)|^2}\left(cg_F^{ij}F_{i,j}+bF\right)(u,t)\\
\equiv&e^{\fr am|F(u,t)|^2}\left(cH_F+bF\right)(u,t).
\end{align*}
This shows that $F(u,t)$ is a solution of the flow \eqref{flow'}.

Conversely, for a given solution $F=F(u,t)$ of the flow \eqref{flow'}, we can similarly find another time-dependent vector field $W(F)$ with the corresponding one-parameter transformations $\hat\vfi_t$. Then we obtain a family of immersions $\hat F(u,t)=F(\hat\vfi(u,t),t)$ which solve the DeTurck mean curvature flow \eqref{dtmcf}.

The above argument and the uniqueness of the solution of \eqref{dtmcf} gives the following existence and uniqueness theorem:

\begin{thm}\label{exiuni}
For any immersion $F_0:M^m\to\bbr^{m+p}$, there exists a maximal $T$: $0<T\leq+\infty$ with a unique smooth solution $F:M^m\times[0,T)\to\bbr^{m+p}$ to the flow \eqref{flow'}.
\end{thm}

\section{Basic evolution formulae}

From now on, we shall use $(x^A)$ to denote the standard coordinates on the standard Euclidean space $(\bbr^{m+p},\ol g)$. For our readers' convenience, we are to derive in this section the basic evolution formulas, with respect to the flow \eqref{flow'}, for the induced metric $g$, the second fundamental form $h$, the mean curvature $H$, and so on.

For a given $T$, $0<T\leq+\infty$, let $F:M^m\times[0,T)\to\bbr^{m+p}$ be a smooth map such that, for each $t\in[0,T)$, $F_t:M^m\to\bbr^{m+p}$ is an immersion. Then the pull-back bundle $F^*T\bbr^{m+p}\to M\times[0,T)$ can be decomposed into two subbundles orthogonal to each other: the tangential part ${\mathcal T}=F_{t*}(TM)$ and the normal part ${\mathcal N}=T^\bot_{F_t}M$. By $F_{t*}$ and ${\mathcal T}$, we can define a ``horizontal distribution\,'' $\mathcal H$ on $M\times[0,T)$, which has also been defined as ${\mathcal H}=\{v\in T(M\times[0,T));\ dt(v)=0\}$ (see \cite{a-b} or \cite{b}). Then, according to \cite{a-b}, there are connections $\nabla$ on $\mathcal H$ and $\nabla^\bot$ on $\mathcal N$, respectively, naturally induced by projections from the pull-back connection $\nabla^{F^*T\bbr^{m+p}}$. In particular, these two connections are both compatible to the relevant bundle metrics.

Now we suppose that $F$ is a solution of \eqref{flow'}. Fix a local coordinate system $(u^i)$ on $M$ and let $\{e_\alpha\}$ be an orthonormal normal frame field of $F(\cdot,t)$. Denote
$$e_i=\pp{}{u^i},\quad F_i\equiv\pp{F}{u^i}=F_*(e_i),\quad g_{ij}=\lagl F_i,F_j\ragl,\quad (g^{ij})=(g_{ij})^{-1},
$$
and
\begin{align*}
&\ol\nabla_{e_j}(F_*e_i)=\sum\Gamma^k_{ij}(F_*e_k)+\sum h^\alpha_{ij}e_\alpha,\\
\nabla_te_i:=\nabla&_{\pp{}{t}}e_i=\sum\Gamma^j_{it}e_j,\quad \nabla^\bot_te_\alpha:=\nabla^\bot_{\pp{}{t}}e_\alpha=\sum\Gamma^\beta_{\alpha t}e_\beta,
\end{align*}
where $\Gamma^j_{it}$ is given by
$$
\ol\nabla_{\pp{}{t}}F_*e_i=\sum\Gamma^j_{it}F_j+\sum\Gamma^\alpha_{it}e_\alpha.
$$
Since
\begin{align}
\ol\nabla_t(F_*e_i)=&\ol\nabla_{e_i}\left(F_*\pp{}{t}\right)+F_*\left(\Big[\pp{}{t},e_i\Big]\right)
=\ol\nabla_{e_i}\left(\fta\right)\nnm\\
=&\fr am\eamf|F|^2_i(cH+bF)+e^{\fr am|F|^2} (c\ol\nabla_{e_i}H+b\ol\nabla_{e_i}F)\nnm\\
=&\eamf\left(\fr{ab}m|F|^2_iF^\top-cA_H(e_i)+bF_i\right) +\eamf\left(\fr{ab}m|F|^2_iF^\bot+\fr{ac}m|F|^2_iH+cH_{,i}\right),\label{4.1}
\end{align}
we have
\begin{align}
F_*(\nabla_t e_i)=&(\ol\nabla_t(F_*e_i))^\top=\eamf\left(\fr{ab}m|F|^2_iF^\top-cA_H(e_i)+bF_i\right)\nnm\\
=&\eamf\left(\left(\fr{ab}{2m}|F|^2_i|F|^2_k-c\sum H^\alpha h^\alpha_{ik}\right)g^{kj}+b\delta^j_i\right)F_j,\label{eit}
\end{align}
or, equivalently,
\be\label{gamat}
\Gamma^j_{it}=\eamf\left(\left(\fr{ab}{2m}|F|^2_i|F|^2_k-c\sum H^\alpha h^\alpha_{ik}\right)g^{kj}+b\delta^j_i\right).
\ee
By \eqref{4.1} we also have
\begin{align}
\ol\nabla_t e_\alpha=&\sum\lagl\ol\nabla_t e_\alpha,F_i\ragl g^{ij} F_j+\sum\lagl\ol\nabla_t e_\alpha,e_\beta\ragl e_\beta\nnm\\
=&-\sum\lagl e_\alpha,\left(\ol\nabla_t F_*(e_i)\right)^\bot\ragl g^{ij} F_j+\sum\Gamma^\beta_{\alpha t}e_\beta\nnm\\
=&-\eamf\left(\fr{ab}m|F|^2_i\lagl F,e_\alpha\ragl+\fr{ac}m|F|^2_i H^\alpha+cH^\alpha_{,i}\right)g^{ij}F_j+\Gamma^\beta_{\alpha t}e_\beta.\label{ealpt}
\end{align}

Moreover, by a direct computation we find that
\begin{align}
\pp{}{t}g_{ij}=\eamf\Big(\frac{ab}m|F|^2_i|F|^2_j-2c\lagl H,h_{ij}\ragl+2bg_{ij}\Big).\label{evogij}
\end{align}
It follows that
\begin{align}
\pp{}{t}g^{ij}=-\eamf g^{ik}g^{lj}\Big(\frac{ab}m|F|^2_{k}|F|^2_{l}-2c\lagl H,h_{kl}\ragl+2bg_{kl}\Big).\label{evogij1}
\end{align}

To obtain the evolution of the second fundamental form $h$, we first find
\begin{align}\label{hijt2}
\nabla^\bot_t h_{ij}=&\left(\pp{}{t}F_{i,j}\right)^\bot=\left(\pp{}{t}F_{ij}-\Gamma^k_{ij}\pp{}{t}F_k\right)^\bot\nnm\\
=&\left(\left(\pp{F}{t}\right)_{ij}-\Gamma^k_{ij}\left(\pp{F}{t}\right)_k\right)^\bot=\left(\left(\pp{F}{t}\right)_{i,j}\right)^\bot\nnm\\
=&\left((\fta)_{,ij}\right)^\bot\nnm\\
=&\Bigg(\bigg(\eamf\Big(\frac{2ac}m\lagl F,F_i\ragl H+\frac{2ab}m\lagl F,F_i\ragl F+cH_i+bF_i\Big)\bigg)_{,j}\Bigg)^\bot\nnm\\
=&\bigg(\frac{2a}m\eamf\lagl F,F_j\ragl\Big(\frac{2ac}m\lagl F,F_i\ragl H+\frac{2ab}m\lagl F,F_i\ragl F+cH_i+bF_i\Big)\bigg)^\bot\nnm\\
&+\bigg(\eamf\Big(\frac{2ac}m\big(\lagl F_j,F_i\ragl+\lagl F,F_{i,j}\ragl\big)H
+\frac{2ac}m\lagl F,F_i\ragl H_j+\frac{2ab}m\big(\lagl F_j,F_i\ragl+\lagl F,F_{i,j}\ragl\big)F\nnm\\
&\qquad+\frac{2ab}m\lagl F,F_i\ragl F_j+cH_{i,j}+bF_{i,j}\Big)\bigg)^\bot\nnm\\
=&\frac{2a}m\eamf\lagl F,F_j\ragl\Big(\frac{2ac}m\lagl F,F_i\ragl H+\frac{2ab}m\lagl F,F_i\ragl F^\bot+cH_{,i}\Big)\nnm\\
&+\eamf\Big(\frac{2ac}m\big(g_{ij}+\lagl F,h_{ij}\ragl\big)H+\frac{2ac}m\lagl F,F_i\ragl H_{,j}
+\frac{2ab}m\big(g_{ij}+\lagl F,h_{ij}\ragl\big)F^\bot+cH_{,ij}+bh_{ij}\Big)\nnm\\
&-c\eamf\lagl h_{il},H\ragl g^{kl}h_{kj}\nnm\\
=&\frac{a}m\eamf\Big(\frac{a}m|F|^2_i|F|^2_j+2g_{ij}+2\lagl F,h_{ij}\ragl\Big)(cH+bF^\bot)\nnm\\
&+\frac{ac}m\eamf\big(|F|^2_iH_{,j}+|F|^2_jH_{,i}\big)+\eamf\big(cH_{,ij}+bh_{ij}\big)-c\eamf\lagl h_{il},H\ragl g^{kl}h_{kj}.
\end{align}
Then by definition and \eqref{gamat}, we obtain
\begin{align*}
\nabla_t h_{ij}=&\nabla^\bot_th_{ij}-h_{kj}\Gamma^k_{it}-h_{ik}\Gamma^k_{jt}\\
=&\frac{a}m\eamf\Big(\frac{a}m|F|^2_i|F|^2_j+2g_{ij}+2\lagl F,h_{ij}\ragl\Big)(cH+bF^\bot)\\
&+\frac{ac}m\eamf\big(|F|^2_iH_{,j}+|F|^2_jH_{,i}\big)+\eamf\big(cH_{,ij}-bh_{ij}\big)\\
&-\fr{ab}{2m}\eamf(h(\nabla|F|^2_,e_j)|F|^2_i+h(\nabla|F|^2,e_i)|F|^2_j)+c\eamf\sum H^\alpha h^\alpha_{jl}h_{ik}g^{kl}.
\end{align*}
The last equation can also be obtained by the time-like Codazzi equation given in (18) of \cite{b}. Since
\begin{align}
h_{ij,kl}=&h_{kl,ij}+\big((h^\beta_{kl}h^\beta_{pj}-h^\beta_{kj}h^\beta_{pl})h_{mi}+(h^\beta_{il}h^\beta_{pj}
-h^\beta_{ij}h^\beta_{pl})h_{km}\nnm\\
&-h^\beta_{ki}h^\beta_{lp}h_{jm}+h^\beta_{ki}h^\beta_{jp}h_{lm}\big)g^{pm},
\end{align}
we have
\begin{align}
\Delta h_{ij}=&H_{,ij}+H^\beta h^\beta_{jk}h_{il}g^{kl}\nnm\\
&+\big(2h^\beta_{ki}h^\beta_{jp}h_{lm}-h^\beta_{kj}h^\beta_{pl}h_{mi}
-h^\beta_{ij}h^\beta_{pl}h_{km}-h^\beta_{ki}h^\beta_{lp}h_{jm}\big)g^{kl}g^{pm}.
\end{align}
It follows that
\begin{align}
\nabla_t h_{ij}=&c\eamf\Delta h_{ij}+\frac{a}m\eamf\Big(\frac{a}m|F|^2_i|F|^2_j+2g_{ij}+2\lagl F,h_{ij}\ragl\Big)(cH+bF^\bot)\nnm\\
&+\frac{ac}m\eamf\big(|F|^2_iH_{,j}+|F|^2_jH_{,i}\big)-b\eamf h_{ij}\nnm\\
&-\fr{ab}{2m}\eamf(h(\nabla|F|^2_,e_j)|F|^2_i+h(\nabla|F|^2_,e_i)|F|^2_j)\nnm\\
&+c\eamf\big(h^\beta_{kj}h^\beta_{pl}h_{mi}+h^\beta_{ij}h^\beta_{pl}h_{km}+h^\beta_{ki}h^\beta_{lp}h_{jm}
-2h^\beta_{ki}h^\beta_{jp}h_{lm}\big)g^{kl}g^{pm}.\label{ht1}
\end{align}
Thus
\begin{align}
\nabla^\bot_tH=&c\eamf\Delta H+\frac{a}m\eamf\Big(\frac{a}m|\nabla|F|^2|^2+2m+2\lagl F,H\ragl\Big)(cH+bF^\bot)\nnm\\
&+\frac{2ac}m\eamf\nabla^\bot_{\nabla|F|^2}H-b\eamf H\nnm\\
&-\fr{ab}m\eamf h(\nabla|F|^2,\nabla|F|^2)+c\eamf\sum H^\alpha h^\alpha_{jl}h_{ik}g^{ij}g^{kl}.\label{Ht2}
\end{align}

To make things simple, we shall call a quantity $C$ to be $g$-{\em constant} (resp. $g$,$F$-{\em constant}) if it can be expressed as a polynomial of the induced metric $g$ (resp. of $g$ and $F$) with constant coefficients. Then, for two given (possibly $\bbr^{m+p}$-valued) tensors $S$ and $T$, we can follow the convention of Hamilton (\cite{ha82}) and Huisken (\cite{hui84}) to simply denote by $S*T$ any $g$-constant-linear combination of tensors formed by contracting the tensor product of $S$ and $T$ w.r.t. $g$, and/or the standard inner product on $\bbr^{m+p}$. Moreover, we shall always write, accordingly, $h^2$, $h^3$, $(\nabla h)^2$, $(\nabla h)^3$ and so on for $h*h$, $h*h*h$, $\nabla h*\nabla h$, $\nabla h*\nabla h*\nabla h$ and so on. Thus by \eqref{ht1} we can write
\begin{align}\label{ht2}
\nabla_th^\alpha_{ij}=&c\eamf\Delta h^\alpha_{ij}+\eamf\left(\nabla|F|^2\cdot\nabla|F|^2 +F^\bot*h+g\right)_{ij}(H^\alpha+bF^\alpha)\nnm\\ &+b\eamf\left((\nabla|F|^2*h)\cdot\nabla|F|^2+h\right)^\alpha_{ij}+\eamf\left(\nabla|F|^2\cdot\nabla H+h^3\right)^\alpha_{ij}\nnm\\
\equiv&c\eamf\Delta h^\alpha_{ij}+\eamf\big((\nabla|F|^2)^2*h+\nabla|F|^2*\nabla h+F^\bot*h^2+g*h +h^3\big)^\alpha_{ij}\nnm\\
&+b\eamf\left((\nabla|F|^2)^2+F^\bot*h+g\right)_{ij}F^\alpha,
\end{align}
where we write $F^\alpha\equiv\lagl F,e_\alpha\ragl$ as the normal component of the position vector $F$. Furthermore, since
$$\Delta|h|^2=2\sum g^{ik}g^{jl}h^\alpha_{kl}\Delta h^\alpha_{ij}+2|\nabla h|^2,$$
it follows that
\begin{align}
\pp{}{t}|h|^2=&c\eamf\Delta|h|^2-2c\eamf|\nabla h|^2\nnm\\
&+\eamf\big((\nabla|F|^2)^2*h^2+\nabla|F|^2*h*\nabla h+F^\bot*h^3+h^2+h^4\big)\nnm\\
&+b\eamf\big((\nabla|F|^2)^2*(F^\bot*h)+(F^\bot*h)^2+F^\bot *h\big)\nnm\\
=&\eamf\left(c\Delta|h|^2-2c|\nabla h|^2+C_1*\nabla h+C_2\right),\label{th2}
\end{align}
where $C_1$ and $C_2$ are polynomials of $g,F,\nabla F$ and $h$.

\section{Some of the key results used for the main theorem}

In this section we shall provide some basic but important facts for the flow \eqref{flow'}, from which one can further make a close observation on how this flow behaves generally and then find more properties of it. In fact, some of the conclusions given in this section and ones given in the next section will be used later as the key part in the proof of the main theorem of the present paper.

Let $F:M^m\times[0,T)\to\bbr^{m+p}$ be a smooth solution of the flow \eqref{flow'} with $[0,T)$ being the maximal time interval of existence, where $T\leq+\infty$.

First of all, from \eqref{flow'} the following basic evolution formula is easily derived:
\be\label{|F|^2}
\pp{}{t}|F|^2=\eamf\big(c\Delta|F|^2+2(b|F|^2-mc)\big),
\ee
by which we can prove
\begin{prop}\label{prop4.1}
Let $F:M^m\times[0,T)\to\bbr^{m+p}$ be a solution of the flow \eqref{flow'}. If $(\frac cb)'\geq 0$ and $|F_0|^2-\fr cb(0)m<0$, then for any $\veps\in(0,\fr cb(0)m-\max|F_0|^2)$, we have $|F_t|^2-\fr cbm<-\veps$ for all $t>0$ as long as $F_{t'}$ keeps an immersion for each $t'\in[0,t)$.
\end{prop}

\begin{proof}
We prove this proposition by contradiction. In fact, by the compactness of $M^m$ we know that
$$|F_0|^2-\fr cb(0)m\leq\max|F_0|^2-\fr cb(0)m<-\veps\text{\ \ on $M^m$}.$$
If there exist $u_0\in M^m$ and $t_0>0$ such that
$$\left(|F|^2-\frac cbm\right)(u_0,t_0)\geq-\veps,$$
then we can find the smallest $t_1\in(0,t_0]$ such that, for this $t_1$,
there exists a $u_1\in M^m$ satisfying
$$\left(|F|^2-\frac cbm\right)(u_1,t_1)=-\veps, \quad \left(|F|^2-\frac cbm\right)\mid_{M\times [0,t_1)}<-\veps.$$
Hence,
$$\left(\pp{}{t}|F|^2-m\left(\frac cb\right)'\right)(u_1,t_1)\geq 0, \quad \left(\Delta|F|^2\right)(u_1,t_1)\leq 0.$$
So at the point $(u_1,t_1)$ we have
\begin{align*}
0\leq&m\left(\frac cb\right)'\leq\pp{}{t}|F|^2\\=&\eamf\left(c\Delta|F|^2+2b\left(|F|^2-\frac cbm\right)\right)\\
=&\eamf(c\Delta|F|^2-2b\veps)<0.\\
\end{align*}
It's a contradiction and the proposition is proved.
\end{proof}

Similarly, we can prove
\begin{prop}\label{prop4.2}
Let $F:M^m\times[0,T)\to\bbr^{m+p}$ be a solution of the flow \eqref{flow'}. If $\left(\frac cb\right)'\leq 0$ and $|F_0|^2-\fr cb(0)m>0$, then for any $\veps\in(0,\min|F_0|^2-\fr cb(0)m)$, we have $|F_t|^2-\fr cbm>\veps$ for all $t>0$ as long as $F_{t'}$ keeps an immersion for each $t'\in[0,t)$.
\end{prop}

The following corollary comes directly from Proposition \ref{prop4.1} and Proposition \ref{prop4.2}:
\begin{cor}\label{cor1}
Let $F:M^m\times[0,T)\to\bbr^{m+p}$ be a solution of the flow \eqref{flow'}, and suppose that $\frac cb=const$.

(1) If $|F_0|^2<\fr cbm$, then $|F_t|^2<\frac cbm$ for all $t>0$ as long as $F_{t'}$ keeps an immersion for each $t'\in[0,t)$;

(2) If $|F_0|^2>\fr cbm$, then $|F_t|^2>\frac cbm$ for all $t>0$ as long as $F_{t'}$ keeps an immersion for each $t'\in[0,t)$.
\end{cor}

The next lemma will be used for a key observation from which we can obtain that the flow \eqref{flow'} keeps invariant the state of the flowing submanifolds being on standard hyperspheres.

\begin{lem}\label{lem4.4}
Let $F:M^m\times[0,T)\to\bbr^{m+p}$ be a solution of the flow \eqref{flow'}. Then, when viewed as an $\bbr^{m+p}$-valued function, the gradient vector $\nabla|F|^2$ meets the following evolution equation:
\begin{align}\label{Q}
\frac12\pp{}{t}\nabla|F|^2=&\frac12c\eamf\Delta(\nabla|F|^2)-c\eamf g^{kl}(\nabla|F|^2)^i_{,k}h_{il}\nnm\\
&+\frac12c\eamf g^{ij}\Big(g^{kl}\lagl h(\nabla|F|^2,F_{k}),h_{il}\ragl+\lagl h(\nabla|F|^2,F_i),H\ragl\Big)F_j\nnm\\
&+\eamf\Big(\frac{a}m\big(c\lagl F,H\ragl+b|F|^2\big)+\frac12b-\frac{ab}{2m}|\nabla|F|^2|^2\Big)\nabla|F|^2\nnm\\
&+\frac{a}{2m}\eamf\big(cH+bF\big)|\nabla|F|^2|^2-\frac12c\eamf{\rm Ric}(\nabla|F|^2).
\end{align}
\end{lem}

\begin{proof} The only thing that needs to do is a direct computation.
Since
$$
\frac12\nabla|F|^2=\lagl F,F_i\ragl g^{ij}F_j,
$$
we find
\begin{align}\label{4.3}
\frac12\pp{}{t}\nabla|F|^2=&\Big(\pp{}{t}g^{ij}\Big)\lagl F,F_i\ragl F_j+g^{ij}\Big\lagl\pp{F}{t},F_i\Big\ragl F_j\nnm\\
&+g^{ij}\Big\lagl F,\Big(\pp{}{t}F_i\Big)\Big\ragl F_j+g^{ij}\lagl F,F_i\ragl\Big(\pp{}{t}F_j\Big).
\end{align}
Then
\begin{align}\label{tfi}
\pp{}{t}F_i=&\Big(\pp{F}{t}\Big)_i=\Big(\eamf\big(cH+bF\big)\Big)_i\nnm\\
=&\eamf\Big(\frac{2a}m\lagl F,F_i\ragl\big(cH+bF\big)+\big(cH_i+bF_i\big)\Big).
\end{align}
Inserting \eqref{evogij1} and \eqref{tfi} into \eqref{4.3}, we have
\begin{align}
\frac12\pp{}{t}\nabla|F|^2
=&\eamf\bigg(g^{ik}g^{jl}\Big(2c\lagl H,h_{kl}\ragl-\frac{ab}m|F|^2_{k}|F|^2_{l}-2bg_{kl}\Big)\frac12|F|^2_iF_j\nnm\\
&+b\lagl F,F_i\ragl g^{ij}F_j+g^{ij}\Big(\frac{2ac}m\lagl F,H\ragl+\frac{2ab}m|F|^2\Big)\lagl F,F_i\ragl F_j\nnm\\
&+g^{ij}\Big(c\lagl F,H_i\ragl+b\lagl F,F_i\ragl\Big)F_j+\frac12g^{ij}|F|^2_i\Big(\frac{a}m|F|^2_j(cH+bF)+(cH_j+bF_j)\Big)\bigg)\nnm\\
=&c\eamf\lagl F,H_i\ragl g^{ij}F_j+\eamf\bigg(c g^{ij}\big\lagl h(\nabla|F|^2,F_{i}),H\big\ragl F_j\nnm\\
&+\Big(\frac{a}m\big(c\lagl F,H\ragl+b|F|^2\big)+\frac12b-\frac{ab}{2m}|\nabla|F|^2|^2\Big)\nabla|F|^2\nnm\\
&+\frac{a}{2m}|\nabla|F|^2|^2(cH+bF)+\frac12c(\nabla|F|^2)^iH_i\bigg).\label{4.15}
\end{align}
Since
\begin{align}
\lagl F,H_i\ragl=\lagl F,(H_i)^\bot\ragl+\lagl F^\top,H_i\ragl,
\end{align}
and
\begin{align}
\lagl F,(H_i)^\bot\ragl=&\lagl F,H_{,i}\ragl=\lagl F,g^{kl}h_{kl,i}\ragl=\lagl F,h_{ik,l}\ragl g^{kl}\nnm\\
=&\lagl F^\bot,F_{i,kl}\ragl g^{kl}=\big\lagl F-\lagl F,F_{p}\ragl g^{pq}F_{q},F_{i,kl}\big\ragl g^{kl}\nnm\\
=&\big(\lagl F,F_{i,k}\ragl_{,l}-\lagl F_{l},F_{i,k}\ragl\big)g^{kl}
-g^{pq}g^{kl}\lagl F,F_{p}\ragl\big(\lagl F_{q},F_{i,k}\ragl_{,l}-\lagl F_{q,l},F_{i,k}\ragl\big)\nnm\\
=&\big(\lagl F,F_i\ragl_{,kl}-\lagl F_{k},F_i\ragl_{,l}\big)g^{kl}+g^{pq}g^{kl}\lagl F,F_{p}\ragl\lagl h_{ql},h_{ik}\ragl\nnm\\
=&\frac12|F|^2_{i,kl}g^{kl}+\frac12g^{kl}g^{pq}|F|^2_{p}\lagl h_{ql},h_{ik}\ragl\nnm\\
=&\frac12\Delta(|F|^2_i)+\frac12g^{kl}\big\lagl h(\nabla|F|^2,F_{k}),h_{il}\big\ragl,\nnm\\
\lagl F^\top,H_i\ragl=&-\lagl h(F^\top,F_i),H\ragl=-\frac12\lagl h(\nabla|F|^2,F_i),H\ragl,
\end{align}
we can find
\begin{align}
\lagl F,H_i\ragl=\frac12\Delta\big(|F|^2_i\big)+\frac12g^{kl}\big\lagl h(\nabla|F|^2,F_{k}),h_{il}\big\ragl
-\frac12\big\lagl h(\nabla|F|^2,F_i),H\big\ragl.
\end{align}
Then it follows that
\begin{align*}
\lagl F,H_i\ragl g^{ij}F_j=&\frac12g^{ij}\Delta(|F|^2_i)F_j+\frac12g^{ij}g^{kl}\lagl h(\nabla|F|^2,F_{k}),h_{il}\ragl F_j\nnm\\
&-\frac12g^{ij}\lagl h(\nabla|F|^2,F_i),H\ragl F_j.
\end{align*}
Putting the above equality into \eqref{4.15} we obtain
\begin{align}
\frac12\pp{}{t}\nabla|F|^2
=&\frac12c\eamf g^{ij}\Delta(|F|^2_i)F_j\nnm\\
&+\frac12c\eamf g^{ij}\big(g^{kl}\lagl h(\nabla|F|^2,F_{k}),h_{il}\ragl+\lagl h(\nabla|F|^2,F_i),H\ragl\big)F_j\nnm\\
&+\eamf\bigg(\Big(\frac{a}m(c\lagl F,H\ragl+b|F|^2)+\frac12b-\frac{ab}{2m}|\nabla|F|^2|^2\Big)\nabla|F|^2\nnm\\
&+\frac{a}{2m}|\nabla|F|^2|^2(cH+bF)+\frac12c(\nabla|F|^2)^iH_i\bigg).
\end{align}

On the other hand, viewing $\nabla|F|^2$ as an $\bbr^{m+p}$-valued function on $M^m$, we can compute its Laplacian as follows:
\begin{align}
\Delta(\nabla|F|^2)
=&g^{kl}\left(g^{ij}|F|^2_iF_j\right)_{k,l}
=g^{ij}g^{kl}\big(|F|^2_{i,kl}F_j+|F|^2_{i,k}F_{j,l}+|F|^2_{i,l}F_{j,k}+|F|^2_iF_{j,kl}\big)\nnm\\
=&g^{ij}\Delta(|F|^2_i)F_j+2g^{ij}g^{kl}\big(\nabla^2|F|^2\big)_{ik}h_{jl}+g^{ij}(\nabla|F|^2)_i\Delta\big(F_j\big)\nnm\\
=&g^{ij}\Delta(|F|^2_i)F_j+2g^{kl}(\nabla|F|^2)^i_{,k}h_{il}+(\nabla|F|^2)^iH_i+{\rm Ric}(\nabla|F|^2),
\end{align}
where ${\rm Ric}:TM\to TM$ denotes the field of linear transformations induced by the Ricci tensor, and the following equality has been used:
$$
\Delta (F_j)=\sum g^{kl}F_{j,kl}=H_j+\sum F_ig^{kl}R^i_{kjl}=H_j+{\rm Ric}(F_j).
$$
Therefore we conclude that
\begin{align}
\frac12\pp{}{t}\nabla|F|^2=&\frac12c\eamf\Delta(\nabla|F|^2)-c\eamf g^{kl}(\nabla|F|^2)^i_{,k}h_{il}\nnm\\
&+\frac12c\eamf g^{ij}\Big(g^{kl}\lagl h(\nabla|F|^2,F_{k}),h_{il}\ragl+\lagl h(\nabla|F|^2,F_i),H\ragl\Big)F_j\nnm\\
&+\eamf\Big(\frac{a}m\big(c\lagl F,H\ragl+b|F|^2\big)+\frac12b-\frac{ab}{2m}|\nabla|F|^2|^2\Big)\nabla|F|^2\nnm\\
&+\frac{a}{2m}\eamf\big(cH+bF\big)|\nabla|F|^2|^2-\frac12c\eamf{\rm Ric}(\nabla|F|^2).
\end{align}
\end{proof}

For any given solution $F:M^m\times[0,T)\to\bbr^{m+p}$ to the flow \eqref{flow'}, we can accordingly construct the following parabolic equation for the tangent vector field $X$ which is taken as an $\bbr^{m+p}$-valued function on $M^m$:
\be\label{*}
\pp{}{t}X=c\eamf\Delta X+Q(F,X,\nabla X),
\ee
where, by writing
$$X=\sum X^i\pp{}{u^i},\quad\nabla_iX=\sum X^j_{,i}\pp{}{u^j},$$
the $\bbr^{m+p}$-valued function $Q(F,X,\nabla X)$ is defined by
\begin{align}
Q(F,X,\nabla X):=&-2c\eamf g^{kl}X^i_{,k}h_{il}\nnm\\
&+c\eamf g^{ij}\big(g^{kl}\lagl h(X,F_{k}),h_{il}\ragl+\lagl h(X,F_i),H\ragl\big)F_j\nnm\\
&+\eamf\Big(\frac{2a}m\big(c\lagl F,H\ragl+b|F|^2\big)+b-\frac{ab}{m}|X|^2\Big)X\nnm\\
&+\frac{a}{m}\eamf\big(cH+bF\big)|X|^2-c\eamf{\rm Ric}(X).
\end{align}
Clearly, $Q(F,X,\nabla X)$ is smooth in $F$, $X$ and satisfies that $Q(F,X,\nabla X)\mid_{X\equiv0}=0$.

From \eqref{Q}, it is easily seen that both $X\equiv 0$ and $X=\nabla|F|^2$ are solutions of \eqref{*}. Then the uniqueness of the solution to the parabolic equation \eqref{*} directly gives the following proposition:
\begin{prop}\label{prop4.4}
Let $F:M^m\times[0,T)\to\bbr^{m+p}$ be a solution of the flow \eqref{flow'}. If the initial submanifold $F_0:M^m\to\bbr^{m+p}$ lies on a standard hypersphere centered at the origin, then $F_{t}(M^m)$ always lies on a likewise standard hypersphere for all $t\in(0,T)$ as long as $F_{t'}$ keeps an immersion for each $t'\in[0,t]$.
\end{prop}

Suppose that the initial submanifold $F_0:M^m\to\bbr^{m+p}$ satisfies $\nabla|F_0|^2\equiv 0$,
that is $F_0(M^m)\subset\bbs^{m+p-1}(R_0)$, where $R_0=|F_0|>0$. Let $\delta>0$ be such that $F_t$ is an immersion for all $t\in[0,\delta)$. Then by Proposition \ref{prop4.4}, $R:=|F|$ is only a function of the time $t$.
From \eqref{|F|^2}, we have an ODE as
\be\label{R}
(R^2)'=2e^{\frac{a}mR^2}(bR^2-mc).
\ee

Now assume that both $a$ and $b$ are constant. Then from \eqref{R} it follows that
\be
e^{-\frac{a}mR^2}(R^2)'=2(bR^2-mc)\nnm
\ee
or, equivalently,
\be
\left(e^{-\frac{a}mR^2}\right)'=-\frac{2ab}m\left(R^2-\fr cbm\right).\nnm
\ee

(1) If $\left(R^2-\fr cbm\right)(0)<0$ and $c'\geq 0$ then, by Proposition \ref{prop4.1}, for any $\veps\in(0,\fr{c(0)}bm-R^2_0)$,
$$\left(R^2-\fr cbm\right)(t)<-\veps <0, \quad \forall t\in[0,\delta).$$
Hence
$$\left(e^{-\frac{a}mR^2}\right)'=-\frac{2ab}m\left(R^2-\fr cbm\right)>\frac{2ab}m\veps>0.$$
Integrate two sides of the above inequality on $[0,t]$, we have
$$e^{-\frac{a}mR^2}-e^{-\frac{a}mR_0^2}>\frac{2ab}m\veps t,$$
that is,
\be\label{R_1}
e^{-\frac{a}mR^2}>\frac{2ab}m\veps t+e^{-\frac{a}mR_0^2},
\ee
implying that
\be\label{T_1}
1-e^{-\frac{a}mR_0^2}\geq e^{-\frac{a}mR^2}-e^{-\frac{a}mR_0^2}>\frac{2ab}m\veps t.
\ee
Therefore, we obtain that
$$
t<\fr m{2\veps ab}\left(1-e^{-\frac{a}mR_0^2}\right),\quad\forall\veps\in\left(0,\fr{c(0)}bm-R^2_0\right).
$$
Let $\veps\to\fr{c(0)}bm-R^2_0$, we get
\be\label{T1}
t\leq T_1:=\fr m{2ab\left(\fr{c(0)}bm-R^2_0\right)}\left(1-e^{-\frac{a}mR_0^2}\right),\quad\forall t\in[0,\delta).
\ee
In particular, we have $\delta\leq T_1$. Moreover, by \eqref{R_1} it also holds that
$$
-\frac{a}mR^2>\log\left(\frac{2\veps ab}mt+e^{-\frac{a}mR_0^2}\right),
\quad\forall\veps\in\left(0,\fr{c(0)}bm-R^2_0\right),$$
or, equivalently, for any $\veps\in(0,\fr{c(0)}bm-R^2_0)$,
$$
R^2<\log\left(\frac{2\veps ab}mt+e^{-\frac{a}mR_0^2}\right)^{-\fr m{a}}.
$$
Taking the limit as $\veps\to \fr{c(0)}bm-R^2_0$, we find
$$
R^2\leq\log\left(\frac{2 ab\left(\fr{c(0)}bm-R^2_0\right)}mt+e^{-\frac{a}mR_0^2}\right)^{-\fr m{a}}.
$$
Therefore, in the case of $\delta=T_1$, it must hold that
$$
\lim\limits_{t\to T_1}R^2=0.
$$

(2) If $\left(R^2-\fr cbm\right)(0)>0$ and $c'\leq 0$ then, by Proposition \ref{prop4.2}, for any $\veps\in\left(0,R^2_0-\fr{c(0)}bm\right)$,
$$\left(R^2-\fr cbm\right)(t)>\veps>0, \quad \forall t\in[0,\delta).$$
So that
$$\left(e^{-\frac{a}mR^2}\right)'=-\frac{2ab}m\left(R^2-\fr cbm\right)<-\frac{2ab}m\veps<0.$$
Integrate two sides of the above inequality on $[0,t]$, we have
\be\label{R_2}
e^{-\frac{a}mR^2}-e^{-\frac{a}mR_0^2}<-\frac{2ab}m\veps t,
\ee
or,
\be\label{T_2}
0<e^{-\frac{a}mR^2}<-\frac{2ab}m\veps t+e^{-\frac{a}mR_0^2}.
\ee
The last inequality shows that
$$
t<\fr m{2\veps ab}e^{-\frac{a}mR_0^2},\quad\forall\veps\in\left(0,R^2_0-\fr{c(0)}bm\right).
$$
Letting $\veps$ tend to $R^2_0-\fr{c(0)}bm$, we find
\be\label{T2}
t\leq T_2:=\fr m{2 ab\left(R^2_0-\fr{c(0)}bm\right)}e^{-\frac{a}mR_0^2},\quad\forall t\in[0,\delta),
\ee
implying that $\delta\leq T_2$. Furthermore, by \eqref{T_2} we also know that, for any $\veps\in(0,R^2_0-\fr{c(0)}bm)$,
$$
-\frac{a}mR^2<\log\left(-\frac{2\veps ab}mt+e^{-\frac{a}mR_0^2}\right),
$$
or, equivalently,
$$
R^2>\log\left(-\frac{2\veps ab}mt+e^{-\frac{a}mR_0^2}\right)^{-\fr m{a}}.
$$
Take the limit as $\veps\to R^2_0-\fr{c(0)}bm$. It follows that
$$
R^2\geq\log\left(-\frac{2 ab\left(R^2_0-\fr{c(0)}bm\right)}mt+e^{-\frac{a}mR_0^2}\right)^{-\fr m{a}}.
$$
Therefore, in the case of $\delta=T_2$, it must hold that
$$
\lim\limits_{t\to T_2}R^2=+\infty.
$$

Due to Proposition \ref{prop4.4}, the above argument has proved the following conclusion:
\begin{thm}\label{thm4.5}
Let $F:M^m\times[0,T)\to\bbr^{m+p}$ be a smooth solution of the flow \eqref{flow'} with both $a$ and $b$ being constant, and $T$ be maximal. Suppose that $\delta\in(0,T]$ is the largest number such that, for all $t\in[0,\delta)$, $F_t$ is an immersion. If the initial submanifold $F_0:M^m\to\bbr^{m+p}$ is contained in a standard hypersphere $\bbs^{m+p-1}(R_0)$ of radius $R_0$ which is centered at the origin, i.e., $|F_0|\equiv R_0$, then $F_t(M^m)$ is always kept on a likewise standard hypersphere. Furthermore,

(1) If $(b|F|^2-mc)(0)<0$ and $c'\geq 0$, then
$$
\delta\leq T_1=\fr m{2ab\left(\fr{c(0)}bm-R^2_0\right)}\left(1-e^{-\frac{a}mR_0^2}\right).
$$
Furthermore, in the case of $\delta=T_1$, $F_t(M^m)$ will be convergent to the origin as $t\to T_1$;

(2) If $(b|F|^2-mc)(0)>0$ and $c'\leq 0$, then
$$
\delta\leq T_2=\fr m{2 ab\left(R^2_0-\fr{c(0)}bm\right)}e^{-\frac{a}mR_0^2}.
$$
Furthermore, in the case of $\delta=T_2$, it must hold that $\lim\limits_{t\to T_2}|F_t|^2=+\infty$;

(3) If  $(b|F|^2-mc)(0)=0$ and $c$ is constant, then $\delta=+\infty$ and $F_t(M^m)$ is kept on the fixed standard hypersphere $\bbs^{m+p-1}(R_0)$ for all $t\geq 0$.
\end{thm}

\begin{rmk}
The conclusion (3) comes directly due to the uniqueness of solutions to the ODE \eqref{R}.
\end{rmk}

\begin{cor}\label{cor4.6}
Let $F:M^{n-1}\times[0,T)\to\bbr^n$ be a smooth solution of the flow \eqref{flow'} with $a,b,c$ all being constant, and $T$ be maximal. Suppose that $\delta\in(0,T]$ is the largest number such that, for all $t\in[0,\delta)$, $F_t$ is an immersion, and that the initial submanifold $F_0:M^{n-1}\to\bbr^n$ of \eqref{flow'} is a standard hypersphere centered at the origin, then $F_t:M^{n-1}\to\bbr^n$ remains a likewise standard hypersphere for each $t\in[0,\delta)$. Moreover,

(1) If $b|F_0|^2<(n-1)c$ and $T_1$ is given by \eqref{T1} with $R_0=|F_0|$ and $m=n-1$, then $\delta=T_1$ and $\lim\limits_{t\to T_1}
|F_t|^2=0$;

(2) If $b|F_0|^2>(n-1)c$ and $T_2$ is given by \eqref{T2} with $R_0=|F_0|$ and $m=n-1$, then $\delta=T_2$ and $\lim\limits_{t\to T_2}
|F_t|^2=+\infty$;

(3) If $b|F_0|^2=(n-1)c$, then $\delta=+\infty$ and $F_t\equiv F_0$ for all $t\geq 0$.
\end{cor}

\begin{proof}
Clearly, by Theorem \ref{thm4.5}, $F_t$ is always a standard hypersphere centered at the origin for all $t\in[0,\delta)$. Note that all the standard hyperspheres do not have any singularity point.

Case (1) $b|F_0|^2<(n-1)c$. Since, for every $t<T_1$, $F_{t}$ remains an immersion, it must hold by Proposition \ref{prop4.1} that $b|F_t|^2<(n-1)c$, which can be in turn viewed as a new initial submanifold of \eqref{flow'}. This allows us to use the conclusion (1) of Theorem \ref{thm4.5} to obtain that $\delta=T_1$ and $\lim\limits_{t\to T_1}|F|^2=0$.

Case (2) $b|F_0|^2>(n-1)c$. Then the similar argument to that in Case (1) will show the conclusion.

Case (3) $b|F_0|^2=(n-1)c$. Then the conclusion comes directly from Theorem \ref{thm4.5}(3).
\end{proof}

In Corollary \ref{cor4.6}, let $n=m+p$, $c=\fr m{m+p-1}$, $b=1$, and $a=\frac{m+p-1}m$. Then, by \eqref{R}, we can obtain the following initial problem of ODE:
\be\label{R'}
\frac12\frac{dR^2}{dt}=e^{\frac1mR^2}(R^2-m),\quad R^2(0)=R_0^2
\ee
for a given positive number $R_0$.

On the other hand, from the flow \eqref{flow}, we also have
$$
\frac12\pp{}{t}|F|^2=\emf(\frac12\Delta|F|^2+|F|^2-m).
$$
It then follows that
\begin{align}
\frac12\pp{}{t}(|F|^2-R^2)=&\frac12\emf\Delta(|F|^2-R^2)+\emf(|F|^2-m)-e^{\frac1mR^2}(R^2-m)\label{a2}\\
=&\frac12\emf\Delta(|F|^2-R^2)+\emf(|F|^2-R^2)+(\emf-e^{\frac1mR^2})(R^2-m)\label{a1},
\end{align}
where $R=R(t)$ satisfies \eqref{R'}.

The above equations \eqref{a2} and \eqref{a1} will be used to prove the following proposition:
\begin{prop}\label{prop4.8}
Let $F:M^m\times[0,\delta)\to\bbr^{m+p}$ and $R^2=R^2(t)$ be determined respectively by \eqref{flow} and \eqref{R'}. Suppose that $F_t$ is an immersion for all $t\in[0,\delta)$. Denote
\be\label{T'}
T_1=\fr m{2(m-R_0^2)}\left(1-e^{-\frac{1}{m}R_0^2}\right),\quad
T_2=\fr m{2(R_0^2-m)}e^{-\frac{1}{m}R_0^2}.
\ee

(1) If $|F_0|^2<m$ then, for any $R^2_0\in\left(\max|F_0|^2,\frac{1}{2}(m+\max|F_0|^2)\right)$ and any $\veps\in(0,R^2_0-\max|F_0|^2)$, it holds that $|F|^2-R^2<-\veps$ for all $t\in[0,\delta)\cap[0, T_1)$;

(2) If $|F_0|^2>m$ then, for any $R^2_0\in\left(\frac{1}{2}(m+\min|F_0|^2),\min|F_0|^2\right)$ and any $\veps\in(0,\min|F_0|^2-R^2_0)$, it holds that $|F|^2-R^2>\veps$ for all $t\in[0,\delta)\cap[0,T_2)$.
\end{prop}

\begin{proof} (1) If $|F_0|^2<m$, then by the compactness and connectedness of $M$, we have $\max|F_0|^2<m$.
For any $R^2_0\in\left(\max|F_0|^2,\frac{1}{2}(m+\max|F_0|^2)\right)$ and any $\veps\in(0,R^2_0-\max|F_0|^2)$, we have $|F_0|^2-R_0^2\leq\max|F_0|^2-R_0^2<-\veps$. Moreover $R^2=R^2(t)$ is well-defined in $[0,T_1)$ due to the argument given before Theorem \ref{thm4.5}, and the continuity of the function $|F|^2-R^2$ assures that $|F|^2-R^2<-\veps$ for sufficiently small $t>0$. On the other hand, as long as we have $|F|^2-R^2<0$ on some interval $[0,t_0)$, then by \eqref{a2} and the Lagrange mean value theorem, we obtain that
$$
\pp{}{t}(|F|^2-R^2)<\emf\Delta(|F|^2-R^2),\quad \forall t\in[0,t_0).
$$
From the maximum principle of parabolic equations, it follows that
$$
|F|^2-R^2\leq\max|F_0|^2-R_0^2<-\veps
$$
on $M^m\times[0,t_0)$.

If the conclusion (1) is wrong, then there exists a first time $t_0>0$ such that $$\max|F_{t_0}|^2-R^2(t_0)=|F|^2(u_0,t_0)-R^2(t_0)=-\veps<0$$
for some point $u_0\in M^m$. Since $|F|^2-R^2$ is a continuous function of the time $t$, there must also be some $t_1>t_0$ such that $|F|^2-R^2<0$ on the whole interval $[0,t_1)$. According to the previous argument, it must hold that $|F|^2-R^2<-\veps$ for all $t\in[0,t_1)$. In particular, we can obtain a strict inequality $|F_{t_0}|^2-R^2(t_0)<-\veps$, which contradicts to the equality that $\max|F_{t_0}|^2-R^2(t_0)=-\veps$ and the conclusion is thus proved.

(2) If $|F_0|^2>m$ then, once again, we can use the compactness and connectedness of $M$ to conclude that $\min|F_0|^2>m$.
For any $R^2_0\in\left(\frac{1}{2}(m+\min|F_0|^2),\min|F_0|^2\right)$ and any $\veps\in(0,\min|F_0|^2-R^2_0)$, we have $|F_0|^2-R_0^2\geq \min|F_0|^2-R_0^2>\veps$. Moreover, the function $R^2=R^2(t)$ is well-defined in $[0,T_2)$ due to the argument given before Theorem \ref{thm4.5}. Now we are able to claim that $|F|^2-R^2>\veps$ for all $t\in[0,\delta)\cap[0,T_2)$. Otherwise there must exist a first time $t_0>0$ and a corresponding point
$u_0\in M^m$ such that
\be\label{adeq1}
\min|F_{t_0}|^2-R^2(t_0)=(|F|^2-R^2)(u_0,t_0)=\veps.
\ee
Note, by the argument given before Theorem \ref{thm4.5}, it follows that $R^2(t_0)-m>0$. Hence by \eqref{a1} and \eqref{adeq1} we have
\begin{align*}
0\geq &\frac12\pp{}{t}(|F|^2-R^2)(u_0,t_0)\\
=&\frac12e^{\frac1m|F|^2(u_0,t_0)}\Delta(|F|^2-R^2)\mid_{(u_0,t_0)}+e^{\frac1m|F|^2(u_0,t_0)}(|F|^2-R^2)(u_0,t_0)\\
&+(\emf-e^{\frac1mR^2})\mid_{(u_0,t_0)}(R^2(t_0)-m)\\
\geq &e^{\frac1m|F|^2(u_0,t_0)}\veps+e^{\frac1mR^2(t_0)}(e^{\frac1m\veps}-1)(R^2(t_0)-m)>0.
\end{align*}
which is impossible! This completes the proof of the conclusion (2).
\end{proof}

Combining Corollary \ref{cor4.6} and Proposition \ref{prop4.8} we know that, the maximal time interval for the existence of the regular solution to the flow \eqref{flow} with $|F_0|^2\neq m$ is not larger than that of the origin-centered standard hyperspheres whose radii satisfy the equation \eqref{R'},
while the maximal existence time of the latter with $|R_0|^2\neq m$ is finite by Corollary \ref{cor4.6}, thus we can easily obtain the finiteness of singular time as follows:

\begin{cor}\label{cor4.9}
Let $F:M^m\times[0,\delta)\to\bbr^{m+p}$ be a solution to the flow \eqref{flow}. Suppose that, for each $t\in[0,\delta)$, $F_t$ is an immersion. If $|F_0|^2\neq m$ everywhere on $M^m$, then $\delta<+\infty$.
\end{cor}

Corollary \ref{cor4.9} will be used for the blow-up argument in the next section.

\section{Higher-order derivative estimates and the blow-up argument}

This section is one of the key parts in the proof of the main theorem of the present paper. First of all, note that the existence and the uniqueness of the maximal solution are given by Theorem \ref{exiuni}. In what follows, after giving some necessary estimates both for the higher order derivatives of the second fundamental form and for those of the solution $F$ itself, we are devoted to prove the following finite time blow-up theorem:

\begin{thm}\label{thm5.1}
Let $F:M^m\times[0,T)\to\bbr^{m+p}$, $T\leq+\infty$, be a maximal regular solution of the curvature flow \eqref{flow} such that $|F_0|^2\neq m$ everywhere. Then $T<+\infty$ and either $\lim\limits_{t\to T}\max|F|^2=+\infty$ or $\lim\limits_{t\to T}\max|h|^2=+\infty$.
\end{thm}

The following Young's inequality is frequently used in our estimation later:

\begin{lem}[Young's inequality]\label{young}
Let $a$ and $b$ be two nonnegative real numbers and $p$ and $q$ be positive real numbers such that $1/p+1/q=1$. Then
$$ab\leq\veps^p{\frac{ a^{p}}{p}}+\frac1{\veps^q}{\frac{b^{q}}{q}},\quad\forall\veps>0.$$
The equality holds if and only if $\veps^{p+q}a^p=b^q$. In particular, we have the following so-called Peter-Paul inequality:
$$2ab\leq\veps a^2+\fr1\veps b^2$$
for any $\veps>0$.
\end{lem}

\begin{lem}\label{lem-nbla} For any $l\geq 1$ it holds that
\be\label{nbla}
\nabla^l\left(\eamf\right)=\eamf\left(\sum_{p=1}^l\sum_{r_1+\cdots+r_p =l}C_{r_1\cdots r_p}\nabla^{r_1}|F|^2\otimes \cdots\otimes\nabla^{r_p}|F|^2\right),
\ee
where $C_{r_1\cdots r_p}$ and $r_1,\cdots,r_p\geq 1$ are constants. In particular,
\be
\left|\nabla^l\left(\eamf\right)\right|^2=e^{\fr{2a}m|F|^2}\sum_{p=2}^{2l}\sum_{r_1+\cdots+r_p=2l}\nabla^{r_1}|F|^2* \cdots*\nabla^{r_p}|F|^2,\quad l\geq 1.
\ee
\end{lem}

\begin{proof} We shall use the method of induction.
For $l=1$, we have
\begin{align*}
\nabla\left(\eamf\right)=&\fr am\eamf\nabla|F|^2\\
\equiv&\eamf\left(\sum_{p=1}^1\sum_{r_1+\cdots+r_p=1}C_{r_1\cdots r_p}\nabla^{r_1}|F|^2\otimes\cdots\otimes\nabla^{r_p}|F|^2\right)
\end{align*}
with $C_1=\fr am $.

Suppose that formula \eqref{nbla} is true for $l=k\geq 1$, that is
$$
\nabla^k\left(\eamf\right)=\eamf\left(\sum_{p=1}^k\sum_{r_1+\cdots+r_p=k}C_{r_1\cdots r_p}\nabla^{r_1}|F|^2\otimes \cdots\otimes\nabla^{r_p}|F|^2\right).
$$
Then for $l=k+1$, we have
\begin{align*}
&\nabla^{k+1}\left(\eamf\right)=\nabla\left(\nabla^k\left(\eamf\right)\right)\\
=&\nabla\left(\eamf\right)\otimes\left(\sum_{p=1}^k\sum_{r_1+\cdots+r_p=k}C_{r_1\cdots r_p}\nabla^{r_1}|F|^2\otimes \cdots\otimes\nabla^{r_p}|F|^2\right)\\
&+\eamf\sum_{p=1}^k\sum_{r_1+\cdots+r_p=k}C_{r_1\cdots r_p}\nabla\left(\nabla^{r_1}|F|^2\otimes\cdots\otimes\nabla^{r_p}|F|^2\right)\\
=&\fr am\eamf\nabla|F|^2\otimes\left(\sum_{p=1}^k\sum_{r_1+\cdots+r_p=k}C_{r_1\cdots r_p}\nabla^{r_1}|F|^2\otimes \cdots\otimes\nabla^{r_p}|F|^2\right)\\
&+\eamf\sum_{p=1}^k\sum_{r_1+\cdots+r_p=k}C_{r_1\cdots r_p}\left(\nabla^{r_1+1}|F|^2\otimes\cdots\otimes\nabla^{r_p}|F|^2 +\cdots+\nabla^{r_1}|F|^2\otimes\cdots\otimes\nabla^{r_p+1}|F|^2\right)\\
=&\eamf\left(\sum_{p=1}^k\sum_{r_1+\cdots+r_p=k}\fr am C_{r_1\cdots r_p}\nabla|F|^2\otimes\nabla^{r_1}|F|^2\otimes \cdots\otimes\nabla^{r_p}|F|^2\right)\\
&+\eamf\sum_{p=1}^k\sum_{r_1+\cdots+r_p=k+1}C_{r_1\cdots r_p}\left(\nabla^{r_1}|F|^2\otimes\cdots\otimes\nabla^{r_p}|F|^2\right)\\
=&\eamf\left(\sum_{p=1}^k\sum_{r_1=1, r_2+\cdots+r_{p+1}=k}C_{r_1\cdots r_{p+1}}\nabla^{r_1}
|F|^2\otimes\nabla^{r_2}|F|^2\otimes\cdots\otimes\nabla^{r_{p+1}}|F|^2\right)\\
&+\eamf\sum_{p=1}^k\sum_{r_1+\cdots+r_p =k+1}C_{r_1\cdots r_p}\left(\nabla^{r_1}|F|^2\otimes\cdots\otimes\nabla^{r_p}|F|^2\right)\\
=&\eamf\left(\sum_{p=2}^{k+1}\sum_{r_1=1, r_2+\cdots+r_p=k}C_{r_1\cdots r_p}\nabla^{r_1}
|F|^2\otimes\nabla^{r_2}|F|^2\otimes\cdots\otimes\nabla^{r_p}|F|^2\right)\\
&+\eamf\sum_{p=1}^k\sum_{r_1+\cdots+r_p=k+1,r_1\geq 2}C_{r_1\cdots r_p}\left(\nabla^{r_1}|F|^2\otimes\cdots\otimes\nabla^{r_p}|F|^2\right)\\
=&\eamf\sum_{p=1}^{k+1}\sum_{r_1+\cdots+r_p=k+1}C_{r_1\cdots r_p}\left(\nabla^{r_1}|F|^2\otimes\cdots\otimes\nabla^{r_p}|F|^2\right),
\end{align*}
where some same symbols have been used to denote different new coefficients. So Lemma \ref{lem-nbla} is proved.
\end{proof}

It is not hard to see that, for any $r\geq 1$ and each $k=1,2,\cdots$, it holds that
\begin{align}
\nabla^r|F|^2=&2\sum_{s=0}^{k-1}C^s_r\lagl\nabla^sF, \nabla^{r-s}F\ragl,\quad\text{for }r=2k-1;\label{nblr|F|2rod}\\
\nabla^r|F|^2=&2\sum_{s=0}^{k-1}C^s_r\lagl\nabla^s F,\nabla^{r-s}F\ragl+C^k_r\lagl\nabla^kF,\nabla^kF\ragl, \quad\text{for }r=2k.
\label{nblr|F|2rev}
\end{align}
Therefore, by Lemma \ref{lem-nbla}, we easily obtain

\begin{lem}\label{lem-nbla1} For any $l\geq 1$, we have
\begin{align}
\nabla^l\left(\eamf\right)=&\eamf\sum_{p=1}^l\sum_{r_1+\cdots+r_p =l}\sum_{0\leq s_1\leq[\fr12r_1]}\!\!\!\cdots\!\!\!\sum_{0\leq s_p\leq [\fr12r_p]}C^{s_1\cdots s_p}_{r_1\cdots r_p}\lagl\nabla^{s_1}F,\nabla^{r_1-s_1}F\ragl\cdots\lagl\nabla^{s_p}F,\nabla^{r_p-s_p}F\ragl\nnm\\
\equiv&\eamf\sum_{p=1}^l\sum_{r_1+\cdots+r_{2p}=l}\nabla^{r_1}F*\cdots*\nabla^{r_{2p}}F\label{nbla1}
\end{align}
where, in the first equality, we have omitted the tensor product sign $\otimes$, and $C^{s_1\cdots s_p}_{r_1\cdots r_p}$ denote certain different constants.
\end{lem}

The following formulae are directly from \cite{l-z}:
\begin{lem}[\cite{l-z}]\label{lem5.5} For any $k\geq 0$,
\begin{align}
\nabla^{l+2}F=&\nabla^l h+\sum_{\iota=0}^{k-1}(*^{2(k-\iota)}_{2\iota+1}h)^iF_*(e_i)+\sum_{\iota=0}^{k-1}(*^{2(k-\iota)+1}_{2\iota}h)^\alpha e_\alpha,\quad\text{if }l=2k;\label{nbla f1}\\
\nabla^{l+2}F=&\nabla^l h+\sum_{\iota=0}^k(*^{2(k-\iota+1)}_{2\iota}h)^iF_*(e_i)+\sum_{\iota=0}^{k-1}(*^{2(k-\iota)+1}_{2\iota+1}h)^\alpha e_\alpha,\quad\text{if }l=2k+1\label{nbla f2}
\end{align}
where, for integers $p\geq 1$ and $q\geq 0$,
$$
*^0_qh=1,\quad*^p_qh=\sum_{r_1+\cdots+r_p=q}\nabla^{r_1}h*\cdots*\nabla^{r_p}h.
$$
\end{lem}

To proceed, we need the following identities which are derived in \cite{a-b} (see also \cite{b}):
\begin{align}
R(\partial_t,e_i,e_j,e_k)=&\lagl\nabla^\bot_{e_k}(F_t)^\bot,h_{ij}\ragl-\lagl\nabla^\bot_{e_j}(F_t)^\bot,h_{ik}\ragl,\label{rtijk1}\\
R^\bot(\partial_t,e_i,e\alpha,e_\beta)=&\lagl\nabla^\bot_{A_\alpha(e_i)}(F_t)^\bot,e_\beta\ragl-\lagl\nabla^\bot_{A_\beta(e_i)}
(F_t)^\bot,e_\alpha\ragl\label{rtiab1}.
\end{align}

Moreover, by the flow equation \eqref{flow'}, the fact that $\nabla^\bot F^\bot=F^\top*h=\nabla|F|^2*h$, we also obtain
\begin{align}
R(\partial_t,e_i,e_j,e_k)
=&\eamf(\nabla|F|^2*(h^2+bF^\bot*h)+h*\nabla h),\label{rtijk2}\\
R^\bot(\partial_t,e_i,e_\alpha,e_\beta)
=&\eamf(\nabla|F|^2*(h^2+bF^\bot*h)+h*\nabla h).\label{rtiab2}
\end{align}

Now we can prove the following proposition:
\begin{prop}
The evolution of the $l$-th covariant derivative of $h$ is of the form
\begin{align}
\nabla_t\nabla^lh=&c\eamf\Delta\nabla^lh+\eamf\Big(\nabla|F|^2*\nabla^{l+1}h+\left(\nabla|F|^2\right)^2*\nabla^lh+F^\bot*\nabla^lh*h\nnm\\
&+g*\nabla^lh+h^2*\nabla^lh+P_1(g,F,\nabla F,h,\nabla h,\cdots,\nabla^{l-1}h)\Big)\nnm\\
&+b\eamf\left(F^\bot*\nabla^lh+P_2(g,F,\nabla F,h,\nabla h,\cdots,\nabla^{l-1}h)\right)F^\bot,\label{f1}
\end{align}
where $P_1$ and $P_2$ are polynomials of $g,F,\nabla F,h,\nabla h,\cdots,\nabla^{l-1}h$ of degrees dependent only on $l$.
\end{prop}

\begin{proof} We shall first prove \eqref{f1} by induction. When $l=0$, we directly use \eqref{ht2} to get
\begin{align}
\nabla_th^\alpha_{ij}=&c\eamf\Delta h^\alpha_{ij}+\eamf\big((\nabla|F|^2)^2*h+\nabla|F|^2*\nabla h+F^\bot*h^2+g*h +h^3\big)^\alpha_{ij}\nnm\\
&+b\eamf\left((\nabla|F|^2)^2+F^\bot*h+g\right)_{ij}F^\alpha,
\end{align}
clearly implying that \eqref{f1} holds for $l=0$.

Suppose that \eqref{f1} holds for $l=k$, i.e.,
\begin{align*}
\nabla_t\nabla^kh=&c\eamf\Delta\nabla^kh+\eamf\Big(\nabla|F|^2*\nabla^{k+1}h+\left(\nabla|F|^2\right)^2*\nabla^kh+F^\bot*\nabla^kh*h\\
&+g*\nabla^kh+h^2*\nabla^kh+P_1(g,F,\nabla F,h,\nabla h,\cdots,\nabla^{k-1}h)\Big)\\
&+b\eamf\left(F^\bot*\nabla^kh+P_2(g,F,\nabla F,h,\nabla h,\cdots,\nabla^{k-1}h)\right)F^\bot
\end{align*}
Then for $l=k+1$ we have, by the time-like Ricci identity, \eqref{rtijk2} and \eqref{rtiab2}, that
\begin{align}\label{tk+1h}
\nabla_t\nabla^{k+1}h=&\nabla(\nabla_t\nabla^kh)+\eamf\left(\nabla^kh*(\nabla|F|^2*(h^2+bF^\bot*h)+h*\nabla h)\right)\nnm\\
=&c\nabla\left(\eamf\Delta\nabla^kh\right)+\nabla\left(\eamf\right)\Big(\nabla|F|^2*\nabla^{k+1}h +\left(\nabla|F|^2\right)^2*\nabla^kh+F^\bot*\nabla^kh*h\nnm\\
&+g*\nabla^kh+h^2*\nabla^kh+P_1(g,F,\nabla F,h,\nabla h,\cdots,\nabla^{k-1}h)\Big)\nnm\\
&+\eamf\nabla\Big(\nabla|F|^2*\nabla^{k+1}h+\left(\nabla|F|^2\right)^2*\nabla^kh+F^\bot*\nabla^kh*h\nnm\\
&+g*\nabla^kh+h^2*\nabla^kh+P_1(g,F,\nabla F,h,\nabla h,\cdots,\nabla^{k-1}h)\Big)\nnm\\
&+b\nabla\left(\eamf\right)\left(F^\bot*\nabla^kh+P_2(g,F,\nabla F,h,\nabla h,\cdots,\nabla^{k-1}h)\right)F^\bot\nnm\\
&+b\eamf\nabla\left(F^\bot*\nabla^kh+P_2(g,F,\nabla F,h,\nabla h,\cdots,\nabla^{k-1}h)\right)F^\bot\nnm\\
&+b\eamf\left(F^\bot*\nabla^kh+P_2(g,F,\nabla F,h,\nabla h,\cdots,\nabla^{k-1}h)\right)\nabla^\bot F^\bot\nnm\\
&+\eamf\left(\nabla^kh*(\nabla|F|^2*(h^2+bF^\bot*h)+h*\nabla h)\right)\nnm\\
=&\eamf\nabla|F|^2*(\Delta\nabla^kh)+c\eamf\nabla(\Delta\nabla^kh)\nnm\\
&+\eamf\nabla|F|^2*\Big(\nabla|F|^2*\nabla^{k+1}h+\left(\nabla|F|^2\right)^2*\nabla^kh+F^\bot*\nabla^kh*h\nnm\\
&+g*\nabla^kh+h^2*\nabla^kh+P_1(g,F,\nabla F,h,\nabla h,\cdots,\nabla^{k-1}h)\Big)\nnm\\
&+\eamf\Big(\nabla^2|F|^2*\nabla^{k+1}h+\nabla|F|^2*\nabla^{k+2}h+\nabla|F|^2*\nabla^2|F|^2*\nabla^kh
+\left(\nabla|F|^2\right)^2*\nabla^{k+1}h\nnm\\
&+\nabla^\bot F^\bot*\nabla^kh*h+F^\bot*\nabla^{k+1}h*h+F^\bot*\nabla^kh*\nabla h\nnm\\
&+g*\nabla^{k+1}h+h*\nabla h*\nabla^kh+h^2*\nabla^{k+1}h+P_1(g,F,\nabla F,h,\nabla h,\cdots,\nabla^kh)\Big)\nnm\\
&+b\eamf\nabla|F|^2*\left(F^\bot*\nabla^kh+P_2(g,F,\nabla F,h,\nabla h,\cdots,\nabla^{k-1}h)\right)F^\bot\nnm\\
&+b\eamf\left(\nabla^\bot F^\bot*\nabla^kh+F^\bot*\nabla^{k+1}h+P_2(g,F,\nabla F,h,\nabla h,\cdots,\nabla^kh)\right)F^\bot\nnm\\
&+b\eamf\left(F^\bot*\nabla^kh+P_2(g,F,\nabla F,h,\nabla h,\cdots,\nabla^{k-1}h)\right)\nabla^\bot F^\bot\nnm\\
&+\eamf\left(\nabla^kh*(\nabla|F|^2*(h^2+bF^\bot*h)+h*\nabla h)\right).
\end{align}

On the other hand, for any $S\in\Gamma(\otimes^r{\mathcal H^*}\otimes{\mathcal N})$, we have the following formula of commuting the Laplacian and gradient (\cite{l-z}):
\begin{align*}
\nabla_k(\Delta S)^\alpha_{i_1\cdots i_r}
=\Delta(\nabla_k S)^\alpha_{i_1\cdots i_r}+(\nabla S*h^2+S*h*\nabla h )^\alpha_{i_1\cdots i_rk}.
\end{align*}
In particular, by setting $S=\nabla^kh$, we have
$$\nabla_{i_{k+1}}(\Delta\nabla^kh)^\alpha_{iji_1\cdots i_k}
=\Delta(\nabla^{k+1}h)^\alpha_{iji_1\cdots i_{k+1}}+(\nabla^{k+1}h*h^2+\nabla^kh*h*\nabla h)^\alpha_{iji_1\cdots i_{k+1}}.
$$
This together with \eqref{tk+1h} and the following equalities
\begin{align*}
&F^\top=\sum g^{ij}\lagl F,F_i\ragl F_j=\fr12\sum g^{ij}\nabla_i|F|^2F_j,\\
&F^\top*h=\nabla|F|^2*h,\quad \nabla^2F=h,\quad\nabla^2|F|^2=g+\lagl F,h\ragl,\\
&\nabla^\bot F^\bot=\left(\nabla(F-F^\top)\right)^\bot=F^\top*h=\nabla|F|^2*h,
\end{align*}
directly gives
\begin{align*}
\nabla_t\nabla^{k+1}h=&c\eamf\Delta\nabla^{k+1}h+\eamf\Big(\nabla|F|^2*\nabla^{k+2}h +\left(\nabla|F|^2\right)^2*\nabla^{k+1}h+F^\bot*\nabla^{k+1}h*h\\
&+g*\nabla^{k+1}h+h^2*\nabla^{k+1}h+P_1(g,F,\nabla F,h,\nabla h,\cdots,\nabla^kh)\Big)\\
&+b\eamf\left(F^\bot*\nabla^{k+1}h+P_2(g,F,\nabla F,h,\nabla h,\cdots,\nabla^kh)\right)F^\bot.
\end{align*}
Thus \eqref{f1} also holds for $l=k+1$.
\end{proof}

Since $\Delta|\nabla^lh|^2=2\lagl\nabla^lh,\Delta\nabla^lh\ragl+2|\nabla^{l+1}h|^2$,
we easily obtain the following corollary by \eqref{f1}:
\begin{cor}\label{evonhiderh}
Let $l\geq 1$. Then the following evolution formula of $|\nabla^lh|^2$ holds:
\begin{align}
\pp{}{t}|\nabla^lh|^2=&c\eamf\Delta|\nabla^lh|^2-2c\eamf|\nabla^{l+1}h|^2+\eamf\Big(\nabla|F|^2*\nabla^{l+1}h*\nabla^lh
+\left(\nabla|F|^2\right)^2*(\nabla^lh)^2\nnm\\
&+F^\bot*(\nabla^lh)^2*h+g*(\nabla^lh)^2+h^2*(\nabla^lh)^2+P_1(g,F,\nabla F,h,\nabla h,\cdots,\nabla^{l-1}h)*\nabla^lh\Big)\nnm\\
&+b\eamf\left(F^\bot*\nabla^lh+P_2(g,F,\nabla F,h,\nabla h,\cdots,\nabla^{l-1}h)\right)F^\bot*\nabla^lh\nnm\\
\equiv&c\eamf\Delta|\nabla^lh|^2-2c\eamf|\nabla^{l+1}h|^2\nnm\\
&+\eamf\left(C_1*\nabla^{l+1}h*\nabla^lh+C_2*(\nabla^lh)^2+C_3*\nabla^lh\right),\label{tlh2}
\end{align}
where $C_1,C_2$ and $,C_3$ are polynomials of $g,F,\nabla F$ and $\nabla^rh$ with $0\leq r\leq l-1$.
\end{cor}

Next, we shall consider the flow \eqref{flow} and make some estimations for all the higher order derivatives of the second fundamental form $h$ and the $\bbr^{m+p}$-valued function $F$.

\begin{prop}\label{prop5.7}
Suppose that the curvature flow \eqref{flow} has a solution on a time interval $[0,\tau]$ where $F_t$ is an immersion for each $t\in[0,\tau]$ with $|F_0|^2\neq m$ everywhere on $M^m$. If $|h|^2$ and $F$ are bounded on $[0,\tau]$, say, $\max\{|h|^2,|F|^2\}\leq C^0_0$ with $C^0_0$ independent of $\tau$, then for each $l\geq 1$, it holds that $|\nabla^lh|^2\leq C^0_l(1+1/t^l)$ for all $t\in(0,\tau]$, where $C^0_l$ is a constant that only depends on $m,l$ and $C^0_0$.
\end{prop}

\begin{proof} The idea of proving Proposition \ref{prop5.7} originally comes from \cite{a-b}, and some detailed argument can also be found in \cite{l-z}.

For each $l\geq 1$ define
$$G_l=t^l|\nabla^l h|^2+lt^{l-1}|\nabla^{l-1}h|^2.$$

We shall prove the proposition by induction on $l$. For the case $l=1$, we can use \eqref{th2} and \eqref{tlh2} to find that
\begin{align}
\pp{}{t}G_1=&\emf\Delta G_1+|\nabla h|^2+t\emf\left(-2|\nabla^2 h|^2+C_1*\nabla^2h*\nabla h+C_2*(\nabla h)^2+C_3*\nabla h\right)\nnm\\
&+\emf\left(-2|\nabla h|^2+C'_1*\nabla h+C'_2\right)\label{gl1}
\end{align}
where, as indicated in Corollary \ref{evonhiderh},  $C_1,C_2,C_3$ and $C'_1,C'_2$ are polynomials of $g,F,\nabla F$ and $h$.

Note that $|\nabla F|^2=m$ and $t\leq\tau<T_0<+\infty$ with some finite positive number $T_0$ (see Corollary \ref{cor4.9}). It follows by the assumption of the proposition that all the norms of $C_1,C_2,C_3,C'_1,C'_2$ are dependent only on $m$ and $C^0_0$. This with the Young's inequality gives that
\begin{align*}
|C_1*&\nabla^2h*\nabla h|\leq 2|\nabla^2 h|^2+c'_1|\nabla h|^2,\quad |C_2*(\nabla h)^2|=|C_2*\nabla h*\nabla h|\leq c'_2|\nabla h|^2,\\
&\quad |C_3*\nabla h|\leq c'_3|\nabla h|^2+c'_4,\quad |C'_1*\nabla h|\leq\veps|\nabla h|^2+c'_5,
\end{align*}
where $c'_1,\cdots,c'_5$ and $\veps$ are positive constants with $\veps<1$. Since $\emf\geq 1$, we obtain from \eqref{gl1} the following estimate:
\be\label{pt g1}
\pp{}{t}G_1\leq\emf\Delta G_1-(1-\veps-tc_1)|\nabla h|^2+c_2 \leq \emf\Delta G_1+c_2,\quad\text{for any\ \ } t\leq\fr{1-\veps}{c_1},
\ee
where $c_1$ and $c_2$ are constants that only depend on $m$ and $C^0_0$. Then it follows from the maximum principle that
$$t|\nabla h|^2\leq G_1\leq C^0_0+c_2t$$
or $$|\nabla h|^2\leq\fr1t G_1\leq\fr{C^0_0}{t}+c_2\leq C^0_1(1+\fr1t)$$
on the interval $(0,\fr{1-\veps}{c_1}]$. When $t>\fr{1-\veps}{c_1}$, we can consider the interval $(t-\fr{1-\veps}{2c_1},t+\fr{1-\veps}{2c_1}]$ of length $\fr{1-\veps}{c_1}$ on which similar argument can give a similar estimation for $G_1$. Since $\veps$ can be chosen fixed, we can cover $(\fr{1-\veps}{c_1}-\delta,\tau]$, $0<\delta\leq\fr{1-\veps}{2c_1}$, with a family of such intervals of a fixed length. Due to the finiteness of $T_0$ and the fact that $\tau<T_0$, this consideration will directly lead to the conclusion for $l=1$.

Now we suppose the conclusion is true for less than or equal to $l-1\geq 1$. Then $|h|^2,\cdots |\nabla^{l-1}h|^2$ are all bounded from above. Once more we use Corollary \ref{evonhiderh} to find
\begin{align}\label{gl1'}
\pp{}{t}G_l=&\emf\Delta G_l+lt^{l-1}|\nabla^l h|^2\nnm\\
&+t^l\emf\big(-2|\nabla^{l+1} h|^2+C_4*\nabla^{l+1}h*\nabla^lh+C_5*(\nabla^lh)^2+C_6*\nabla^lh\big)\nnm\\
&+l(l-1)t^{l-2}|\nabla^{l-1}h|^2\nnm\\
&+lt^{l-1}\emf\big(-2|\nabla^l h|^2+C'_4*\nabla^lh*\nabla^{l-1}h+C'_5*(\nabla^{l-1}h)^2+C'_6*\nabla^{l-1}h\big),
\end{align}
where $C_4,C_5,C_6$ and $C'_4,C'_5,C'_6$ are polynomials of $g,F,\nabla F$ and $\nabla^r h$ with $0\leq r\leq l-1$.

As in the case $l=1$, the above argument and the Young's inequality together give that
\begin{align*}
|C_4*&\nabla^{l+1}h*\nabla^l h|\leq 2|\nabla^{l+1} h|^2+c'_6|\nabla^l h|^2,\quad |C_5*(\nabla^l h)^2|=|C_5*\nabla^l h*\nabla^l h|\leq c'_7|\nabla^l h|^2,\\
&\quad |C_6*\nabla^l h|\leq c'_8|\nabla^l h|^2+c'_9,\quad |C'_4*\nabla^l h*\nabla^{l-1} h|\leq\veps|\nabla^l h|^2+c'_{10},
\end{align*}
where $c'_6,\cdots,c'_{10}$ and $\veps$ are positive constants with $\veps<1$. Since $\emf\geq 1$, we obtain from \eqref{gl1'} the following estimate:
\be
\pp{}{t}G_l\leq\emf\Delta G_l-lt^{l-1}\big(1-\veps-tc_3\big)|\nabla^l h|^2+c_4
\leq\emf\Delta G_l+c_4,\quad t\leq\fr{1-\veps}{c_3},
\ee
for some positive $\veps<1$, where $c_3$ and $c_4$ are constants that only depend on $m$ and $C^0_0$. Thus the maximum principle gives
$$t^l|\nabla^l h|^2\leq G_l\leq c_4t, \text{\ or\ } |\nabla^l h|^2\leq \fr{c_4}{t^{l-1}}\leq C^0_l\Big(1+\fr1{t^{l-1}}\Big)$$
on the interval $(0,\fr{1-\veps}{c_3}]$. When $t>\fr{1-\veps}{c_3}$, we can fix a small $\veps>0$ and consider a family of intervals of fixed length no more than $\fr{1-\veps}{c_3}$ similarly as the case $l=1$ to reach the conclusion for $l$.
\end{proof}

From Lemma \ref{lem5.5} and Proposition \ref{prop5.7} we directly obtain the following corollary:

\begin{cor}\label{cor5.8} If $\max\{|h|^2,|F|^2\}<+\infty$, then there exist constants $C'(l)$ and $C''(l)$ such that
\be\label{5.19}
|\nabla^lh|^2\leq C'(l),\quad |\nabla^lF|^2\leq C''(l),\quad l\geq 0.
\ee
\end{cor}

In order to prove Theorem \ref{thm5.1}, we follow \cite{b} to fix a smooth metric $\mathring g$ on $M$ with the Levi-Civita connection $\mathring\nabla$, which can trivially extend to a time-independent metric on $M\times[0,T)$, still denoted by $\mathring g$. Then we need to use the corresponding induced connections, also denoted by $\mathring\nabla$, on the relevant bundles on $M\times[0,T)$ for some computations.

Since the connection coefficients $\mathring\Gamma^j_{it}$ of the time-dependent connection $\mathring\nabla$ vanishes identically, it follows from \eqref{evogij} that
\begin{align*}
(\mathring\nabla_t g)_{ij}=&\pp{}{t}g_{ij}-g_{kj}\mathring\Gamma^k_{it}-g_{ik}\mathring\Gamma^k_{jt}=\pp{}{t}\lagl F_i,F_j\ragl\\
=&\emf\Big(\frac{1}m|F|^2_i|F|^2_j-2\lagl H,h_{ij}\ragl+2g_{ij}\Big).
\end{align*}
Consequently, for any non-zero vector $v=\sum v^ie_i\in TM$, we have
$$
\fr{(\mathring\nabla_t g)(v,v)}{g(v,v)}=\emf\left(\frac{1}m\fr{g(\nabla|F|^2,v)^2}{g(v,v)}-2\fr{\lagl H, h(v,v)\ragl}{g(v,v)}+2\right)
$$
which with Cauchy inequality implies that
$$
\left|\fr{(\mathring\nabla_t g)(v,v)}{g(v,v)}\right|\leq\emf\Big(\frac{1}m|\nabla|F|^2|^2+2|H||h|+2\Big).
$$
Since
$$\nabla|F|^2=2F^\top,\quad |H|^2\leq\fr1m|h|^2,$$
we know that
$$
\left|\fr{(\mathring\nabla_t g)(v,v)}{g(v,v)}\right|\leq\emf\left(\frac{4}m|F|^2+\fr{m+1}m|h|^2+2\right).
$$
Therefore,
if $F$ and $h$ are bounded, then there is a positive constant $C$ dependent only on the bounds of $F$ and $h$ such that
$$
\left|\pp{}{t}\left(\fr{g(v,v)}{\mathring g(v,v)}\right)\right|=\left|\fr{(\mathring\nabla_t g)(v,v)}{g(v,v)}\fr{g(v,v)}{\mathring g(v,v)}\right|\leq C\fr{g(v,v)}{\mathring g(v,v)},\quad\forall\,v\in TM,
$$
from which it is easily seen that (\cite{b})
\be\label{32}
\fr1{\ol c}\mathring g\leq g\leq\ol c\mathring g
\ee
for some constant $\ol c>0$. Thus any estimation of a length function with respect to the metric $\mathring g$ is equivalent to that with respect to the metric $g$.

Let $\mathring T=\mathring\nabla-\nabla$ be the difference of the two connections. Then $\mathring T\in\Gamma(\mathcal{H}^*\otimes \mathcal{H}^*\otimes\mathcal{H})$. Moreover, for any section $S$ of a bundle constructed from the induced bundle $F^*T\bbr^{m+p}$, the horizontal distribution ${\mathcal H}\subset T(M\times[0,T))$ and the normal subbundle ${\mathcal N}\subset F^*T\bbr^{m+p}$, we have that $\mathring\nabla S-\nabla S=S*\mathring T$.

The following lemma can be found in \cite{l-z}:

\begin{lem}[\cite{l-z}]\label{lem5.9} It holds that
\be
\mathring\nabla^{l+1} F=\nabla^{l+1} F+\sum_{q=1}^l\sum_{p=1}^{l+1-q}(\mathring*^p_{l-p-q+1}\mathring T)*\nabla^q F,\quad l\geq 0.
\label{tdnabla f1}
\ee
where, for integers $p\geq 1$ and $q\geq 0$,
$$
\mathring*^p_q\mathring T=\sum_{r_1+\cdots+r_p=q}\mathring\nabla^{r_1}\mathring T*\cdots*\mathring\nabla^{r_p}\mathring T.
$$
\end{lem}

\begin{lem}\label{lem5.10}
If $|h|^2$ and $F$ are bounded, then for each $l\geq 0$, there is a constant $C(l)$, such that $|\mathring\nabla^l\mathring T|^2\leq C(l)$.
\end{lem}

\begin{proof} To prove the lemma, we need the following formula:
\be\label{tdnablat t}
\mathring\nabla_t\mathring\nabla^l\mathring T=\emf(P_1+P_2*\mathring\nabla^{l-1}\mathring T),\quad l\geq 0,
\ee
where $\mathring\nabla^{-1}=0$, $\mathring\nabla^0=1$, and
$$P_1=P_1(g,F,\nabla F,h,\nabla h,\cdots,\nabla^{l+1}h,\mathring T,\mathring\nabla\mathring T,\cdots,\mathring\nabla^{l-2}\mathring T)$$
is a polynomial of $g,F,\nabla F,h,\nabla h,\cdots,\nabla^{l+1}h,\mathring T,\mathring\nabla\mathring T,\cdots,\mathring\nabla^{l-2}\mathring T$,
while
\be\label{p2}
P_2=(\nabla|F|^2)^3+\nabla|F|^2*h^2+\nabla|F|^2*\id_{TM}+\nabla|F|^2*F^\bot*h+\nabla h*h
\ee
is a fixed polynomial of $g,F,\nabla F,h,\nabla h$.

Formula \eqref{tdnablat t} can be shown by the method of induction.

First we consider $l=0$. By the definition of the connection $\mathring\nabla$ on $M\times[0,T)$, it is easily seen that $\mathring\nabla_t\mathring T^k_{ij}=\pp{}{t}\mathring T^k_{ij}=-\pp{}{t}\Gamma^k_{ij}$. Since for each fixed $t\in[0,T)$, $\pp{}{t}\Gamma^k_{ij}$ is a $(1,2)$-tensor on the Riemannian manifold $(M\times\{t\},g)$ w.r.t. the indices $i,j,k$, by using a normal coordinates $(u^i)$ of the induced metric $g$ and \eqref{evogij} we can find that, at each fixed point $p$ on $M\times\{t\}$,
\begin{align}
\mathring\nabla_t\mathring\nabla^0\mathring T^k_{ij}=&-\pp{}{t}\Gamma^k_{ij}=-\fr12g^{kl}\left(\pp{}{u^j}\pp{g_{il}}{t} +\pp{}{u^i}\pp{g_{jl}}{t}-\pp{}{u^l}\pp{g_{ij}}{t}\right)\nnm\\
=&-\fr12g^{kl}\bigg(\nabla_{\pp{}{u^j}}\Big(\emf\big(\fr1m|F|^2_i|F|^2_l-2\lagl H,h_{il}\ragl+2g_{il}\big)\Big)\nnm\\
&\qquad\qquad+\nabla_{\pp{}{u^i}}\Big(\emf\big(\fr1m|F|^2_j|F|^2_l-2\lagl H,h_{jl}\ragl+2g_{jl}\big)\Big)\nnm\\
&\qquad\qquad-\nabla_{\pp{}{u^l}}\Big(\emf\big(\fr1m|F|^2_i|F|^2_j-2\lagl H,h_{ij}\ragl+2g_{ij}\big)\Big)\bigg)\nnm\\
=&-\fr12g^{kl}\emf\bigg(\fr1m|F|^2_j\Big(\fr1m|F|^2_i|F|^2_l-2\lagl H,h_{il}\ragl+2g_{il}\Big)+\nabla_j\Big(\fr1m|F|^2_i|F|^2_l-2\lagl H,h_{il}\ragl+2g_{il}\Big)\nnm\\
&\qquad\quad+\fr1m|F|^2_i\Big(\fr1m|F|^2_j|F|^2_l-2\lagl H,h_{jl}\ragl+2g_{jl}\Big)+\nabla_i\Big(\fr1m|F|^2_j|F|^2_l-2\lagl H,h_{jl}\ragl+2g_{jl}\Big)\nnm\\
&\qquad\quad-\fr1m|F|^2_l \Big(\fr1m|F|^2_i|F|^2_j-2\lagl H,h_{ij}\ragl+2g_{ij}\Big)-\nabla_l\Big(\fr1m|F|^2_i|F|^2_j-2\lagl H,h_{ij}\ragl+2g_{ij}\Big)\bigg)\nnm\\
=&-\fr12g^{kl}\emf\bigg(\fr{1}{m^2}|F|^2_i|F|^2_j|F|^2_l-\fr2m\Big(|F|^2_j\lagl H,h_{il}\ragl+|F|^2_i\lagl H,h_{jl}\ragl-|F|^2_l\lagl H,h_{ij}\ragl\Big)\nnm\\
&\qquad\quad+\fr2m\Big(|F|^2_jg_{il}+|F|^2_ig_{jl}-|F|^2_lg_{ij}\Big)\nnm\\
&\qquad\quad+\fr1m\Big(|F|^2_{i,j}|F|^2_l+|F|^2_i|F|^2_{l,j}+|F|^2_{j,i}|F|^2_l+|F|^2_j|F|^2_{l,i} -|F|^2_{i,l}|F|^2_j-|F|^2_i|F|^2_{j,l}\Big)\nnm\\
&\qquad\quad-2\Big(\lagl H_{,j},h_{il}\ragl+\lagl H,h_{ilj}\ragl+\lagl H_{,i},h_{jl}\ragl+\lagl H,h_{jli}\ragl
-\lagl H_{,l},h_{ij}\ragl-\lagl H,h_{ijl}\ragl\Big)\bigg)\nnm\\
=&\emf\Big((\nabla|F|^2)^3+\nabla|F|^2*h^2+\nabla|F|^2*\id_{TM}+\nabla|F|^2*F^\bot*h+\nabla h*h\Big)^k_{ij}\label{nabla0t}
\end{align}
where in the last equality we have used the identity that $|F|^2_{i,j}=2(g_{ij}+\lagl F^\bot,h_{ij}\ragl)$. By the arbitrariness of the point $p$, we finally obtain from \eqref{nabla0t} that
\be\label{nabla0t1}
\mathring\nabla_t\mathring\nabla^0\mathring T=\emf((\nabla|F|^2)^3+\nabla|F|^2*h^2+\nabla|F|^2*\id_{TM}+\nabla|F|^2*F^\bot*h+\nabla h*h).
\ee
Thus \eqref{tdnablat t} holds for $l=0$ with
$$
P_1=(\nabla|F|^2)^3+\nabla|F|^2*h^2+\nabla|F|^2*\id_{TM}+\nabla|F|^2*F^\bot*h+\nabla h*h.
$$

Suppose that \eqref{tdnablat t} holds for $l=k\geq 0$, i.e.,
\be
\mathring\nabla_t\mathring\nabla^k\mathring T=\emf(P_1+P_2*\mathring\nabla^{k-1}\mathring T),
\ee
where
$P_1$ is a polynomial of $g,F,\nabla F,h,\nabla h,\cdots,\nabla^{k+1}h,\mathring T,\mathring\nabla\mathring T,\cdots,\mathring\nabla^{k-2}\mathring T$, and
$P_2$ is given by \eqref{p2}. Then for $l=k+1$, we compute
\begin{align}\label{tdnabk+1t}
\mathring\nabla_t\mathring\nabla^{k+1}\mathring T=&\mathring\nabla(\mathring\nabla_t\mathring\nabla^k\mathring T)
=\mathring\nabla\left(\emf\right)(P_1+P_2*\mathring\nabla^{k-1}\mathring T)+\emf\big((\nabla+\mathring T)*P_1\nnm\\
&+\big((\nabla+\mathring T)*P_2\big)*\mathring\nabla^{k-1}\mathring T+P_2*\mathring\nabla\mathring\nabla^{k-1}\mathring T\big)\nnm\\
=&\fr1m\emf\nabla|F|^2(P_1+P_2*\mathring\nabla^{k-1}\mathring T)
+\emf\big(\nabla P_1+\mathring T*P_1+\nabla P_2*\mathring\nabla^{k-1}\mathring T\nnm\\
&+\mathring T*P_2*\mathring\nabla^{k-1}\mathring T+P_2*\mathring\nabla\mathring\nabla^{k-1}\mathring T\big).
\end{align}
It follows easily from \eqref{tdnabk+1t} that
$$
\mathring\nabla_t\mathring\nabla^{k+1}\mathring T=\emf(P'_1+P_2*\mathring\nabla^{k+1-1}\mathring T),
$$
where $P'_1$ is a polynomial of $g,F,\nabla F,h,\nabla h,\cdots,\nabla^{k+1+1}h,\mathring T,\mathring\nabla\mathring T,\cdots,
\mathring\nabla^{k+1-2}\mathring T$, and $P_2$ is given by \eqref{p2}. Therefore, \eqref{tdnablat t} holds for $l=k+1$ and thus holds for all $l\geq 0$ by induction.

Now we can use \eqref{tdnablat t} to complete the proof of Lemma \ref{lem5.10}.

Note that, by Corollary \ref{cor5.8}, if $|h|^2$ and $F$ are uniformly bounded from above for $t\in[0,T)$, then all of $\nabla^lh$ and $\nabla^lF$ are also uniformly bounded for $l\geq 0$. On the other hand, by \eqref{gamat}, we know that $\nabla_t$ and $\mathring\nabla_t$ are related by
\be\label{ad5.25}
\nabla_tS=\mathring\nabla_tS+\emf\left(h^2+(\nabla|F|^2)^2+1\right)*S
\ee
for any bundle-valued tensor $S$ on $M\times[0,T)$. Then, from \eqref{nabla0t1} and Young's inequality, it follows that, when $l=0$,
\begin{align*}
\pp{}{t}|\mathring T|^2=&2\lagl\nabla_t\mathring T,\mathring T\ragl=\left\lagl\mathring\nabla_t\mathring T+\emf\left(h^2 +(\nabla|F|^2)^2+1\right)*\mathring T,\mathring T\right\ragl\\
=&\emf\left(P_1+\left(h^2 +(\nabla|F|^2)^2+1\right)*\mathring T\right)*\mathring T\\
\leq&C_1(1+|\mathring T|^2),
\end{align*}
which implies that
$$\pp{}{t}\log(1+|\mathring T|^2)\leq C_1,\quad\text{or}\quad\pp{}{t}(\log(1+|\mathring T|^2)-C_1t)\leq 0.$$
Since $M$ is compact, we find
$$
\log(1+|\mathring T|^2)-C_1t\leq(\log(1+|\mathring T|^2)-C_1t)|_{t=0}=(\log(1+|\mathring T|^2))|_{t=0}\leq C_2.
$$
So $|\mathring T|^2<1+|\mathring T|^2\leq e^{C_1t+C_2}\leq C(0)$ since the interval $[0,T)$ is bounded.

Suppose that $|\mathring T|^2$, $\cdots$, $|\mathring\nabla^{l-1}\mathring T|^2$ have been shown bounded for $l\geq 1$,
then by \eqref{tdnablat t}, \eqref{ad5.25} and Young's inequality,
\begin{align*}
\pp{}{t}|\mathring\nabla^l\mathring T|^2=&2\lagl\nabla_t\mathring\nabla^l\mathring T,\mathring\nabla^l\mathring T\ragl
=2\left\lagl\mathring\nabla_t\mathring\nabla^l\mathring T+\emf\left(h^2+(\nabla|F|^2)^2+1\right)*\mathring\nabla^l\mathring T,\mathring
\nabla^l\mathring T\right\ragl\\
=&\emf\left(P_1+P_2*\mathring\nabla^{l-1}\mathring T+\left(h^2+(\nabla|F|^2)^2+1\right)*\mathring\nabla^l\mathring T\right)*\mathring
\nabla^l\mathring T
\leq C_3(1+|\mathring\nabla^l\mathring T|^2).
\end{align*}
Thus, as in the case of $l=0$, there exists a $C(l)>0$ such that
$|\mathring\nabla^l\mathring T|^2\leq C(l)$.

Thus the principle of induction finishes the proof.
\end{proof}

Combining Corollary \ref{cor5.8}, Lemma \ref{lem5.9}, Lemma \ref{lem5.10} and the comparability  \eqref{32} of $g$ with $\mathring g$ we have proved the following boundedness result:

\begin{prop}\label{prop5.11} If $|h|^2$ and $F$ are bounded as $t\to T$, then for each $l\geq 0$, there exists a constant $C(l)>0$ such that
$|\mathring\nabla^l F|^2_{\mathring g}\leq\td C(l)|\mathring\nabla^l F|^2_g\leq C(l)$ where $\td C(l)>0$.
\end{prop}

To prove Theorem \ref{thm5.1}, we also need the following formula:
\begin{lem}\label{lem5.13} It holds that
\be\label{5.25}
\pp{}{t}\mathring\nabla^lF=\sum_{p+q=l}\left(\sum_{r+s+t=p}(\mathring*^r_s\mathring T)*\nabla^t \left(\emf\right)\right)*\left(\sum_{r+s+t=q+2,\,t\geq 1}(\mathring*^r_s\mathring T)*\mathring\nabla^t F+\mathring\nabla^q F\right),\quad l\geq 0.
\ee
\end{lem}

\begin{proof} Take $H$ as a section of the induced bundle $F^*T\bbr^{m+p}$. Then we have the following two equations (see equations (5.24) and (5.25) in \cite{l-z}):
\begin{align}
&\mathring\nabla^p\left(\emf\right)=\sum_{r+s+t=p}(\mathring*^r_s\mathring T)*\nabla^t\left(\emf\right),\quad p\geq 0,\label{5.26}\\
&\mathring\nabla^qH=\sum_{r+s+t=q+2,\,t\geq1}(\mathring*^r_s\mathring T)*\mathring\nabla^t F,\quad q\geq 0.\label{5.27}
\end{align}
Thus \eqref{5.25} easily follows by
$$
\pp{}{t}\mathring\nabla^lF=\mathring\nabla^l\pp{F}{t}=\mathring\nabla^l\left(\emf(H+F)\right) =\sum_{p+q=l}\mathring\nabla^p\left(\emf\right)*(\mathring\nabla^qH+\mathring\nabla^q F), \quad l\geq 0.
$$
\end{proof}

{\em The proof of Theorem \ref{thm5.1}}.

The finiteness of the time $T$ comes from Corollary \ref{cor4.9}.

If Theorem \ref{thm5.1} is not true, then by Corollary \ref{cor5.8}, all the covariant derivatives $\nabla^lh$ and $\nabla^lF$ will be uniformly bounded from above. This with Lemma \ref{lem-nbla1} shows that $\nabla^l\left(\emf\right)$ are bounded for all $l\geq 0$. On the other hand, by Lemma \ref{lem5.10} and Proposition \ref{prop5.11}, we obtain that $\mathring\nabla^l\mathring T$ and $\mathring\nabla^lF$ are also bounded for all $l\geq 0$. Therefore, by Lemma \ref{lem5.13}, it is easily seen that, for each $l\geq 0$, $|\pp{}{t}\mathring\nabla^lF|$ will be uniformly bounded from above by a constant $C(l)>0$. Consequently we have
$$\left|\mathring\nabla^lF(u,t_1)-\mathring\nabla^lF(u,t_2)\right|\leq\left|\int^{t_2}_{t_1}\pp{}{t}\mathring\nabla^lFdt\right|\leq C(l)|t_1-t_2|,\quad l\geq 0,$$
for all $u\in M^m$ and $t_1,t_2\in (0,T)$. So, due to the Cauchy convergence theorem,
$\mathring\nabla^lF$ will converge uniformly as $t\to T$, implying that $F(\cdot,t)$ will converge in $C^\infty$-topology to a limit immersion $F(\cdot,T):M^m\to\bbr^{m+p}$. By the existence theorem (Theorem \ref{exiuni}) of the short time solution, $[0,T)$ could not be the maximal time interval for the existence of the regular solution, which is a contradiction. Thus Theorem \ref{thm5.1} is proved.

\section{Proof of the main theorem}

Based on the previous discussions, we shall give a proof of the following main theorem in this paper:
\begin{thm}\label{thm6.1}
Let $F_0:M^m\to\bbr^{m+p}$ be an immersed compact submanifold satisfying $|F_0|^2\neq m$ everywhere on $M^m$. Denote
\be\label{t1t2}
\hat T_1:=\fr m{2(m-\max|F_0|^2)}\left(1-e^{-\frac1m\max|F_0|^2}\right),\quad \hat T_2:=\fr m{2(\min|F_0|^2-m)}e^{-\frac1m\min|F_0|^2}.
\ee
Then there exists a maximal $T>0$, and a smooth solution $F:M^m\times[0,T)\to\bbr^{m+p}$ to the flow \eqref{flow}, such that, for each $t\in[0,T)$, $F_t:M^m\to\bbr^{m+p}$ is an immersion. Furthermore,

(1) If $\min|F_0|^2<m$, then we have either
$$T<\hat T_1 \text{\ \ and\ \ }\lim\limits_{t\to T}\max|h|^2=+\infty$$
or
$$T=\hat T_1 \text{\ \ and\ \ }\lim\limits_{t\to \hat T_1}\max|F|^2=0,$$
namely, $F(M^m)$ converges to the origin as $t$ tends to $\hat T_1$;

(2) If $\max|F_0|^2>m$, then we have either $T<\hat T_2$ and one of the following two holds:

(i) $\lim\limits_{t\to T}\max|h|^2=+\infty$;

(ii) $\lim\limits_{t\to T}\max|F|^2=+\infty$;
\\
or $T=\hat T_2$ and $\lim\limits_{t\to\hat T_2}\min|F|^2=+\infty$.
\end{thm}

\begin{proof}
Let $F_0:M^m\to\bbr^{m+p}$ be the initial submanifold given in the theorem. Since $|F_0(p)|^2\neq m$ for each point $p\in M^m$, the connectedness of $M^m$ implies that either $\max|F_0|^2<m$ or $\min|F_0|^2>m$. By the existence theorem (Theorem \ref{exiuni}) and Corollary \ref{cor4.9}, there is a maximal interval $[0,T)$ with $T<+\infty$ such that a regular solution $F:M^m\times[0,T)\to\bbr^{m+p}$ to the flow \eqref{flow} exists. Moreover, from Theorem \ref{thm5.1}, we also know that either $\lim\limits_{t\to T}\max|F|^2=+\infty$ or $\lim\limits_{t\to T}\max|h|^2=+\infty$.

Case (1): If $\max|F_0|^2<m$, we can take $R_0=\max|F_0|$. Then $R^2_0<m$ and the solution $F^S:\bbs^{m+p-1}\times[0,T^S)\to\bbr^{m+p}$ of the ODE \eqref{R'} which comes from \eqref{flow'} with the original $m$ being replaced by $m+p-1$ and with $a=\fr{m+p-1}m$, $b=1$,
$c=\fr m{m+p-1}$, where the initial submanifold is an origin-centered standard hypersphere $S_0$ of radius $R_0$, gives a maximal regular solution $R^2\equiv R^2(t):=|F^S|^2$, $0\leq t<T^S$. Furthermore, by Corollary \ref{cor4.6}, we have $T^S=\hat T_1$ where $\hat T_1$ is given by \eqref{t1t2} and $\lim\limits_{t\to \hat T_1} R^2=0$.

On the other hand, for any $R'^2_0\in\left(\max|F_0|^2,\frac{1}{2}(m+\max|F_0|^2)\right)$ and any $\veps\in(0,R'^2_0-\max|F_0|^2)$, we let $R'^2=R'^2(t)$, $t\in[0,T_1)$, be the unique solution of \eqref{R'}, where $T_1$ is given by \eqref{T'} with $R^2_0$ being replaced by $R'^2_0$. So we can use Proposition \ref{prop4.8} to conclude that $|F|^2-R'^2<-\veps$ for all $p\in M^m$ and $t\in[0,T)\cap[0,T_1)$.
Furthermore, by the standard theory of ordinary differential equations, the solution $R'^2=R'^2(t)$ of \eqref{R'} continuously depends on the initial value $R'^2_0$. Therefore, by taking the limit of the inequality as $R'^2_0\to\max|F_0|^2\equiv R^2_0$, we conclude that $|F|^2-R^2\leq 0$ for all $p\in M^m$ and $t\in[0,T)\cap[0,\hat T_1)$ since $\veps\to 0$  and $T_1\to\hat T_1$ as $R'^2_0\to\max|F_0|^2$.

By discussions of the above two paragraphs, we can claim that $T\leq\hat T_1$.
In fact, since the initial submanifold $F_{0}(M^m)$ of the flow \eqref{flow} is bounded by an origin-centered standard hypersphere $S_0$ of radius $R_0$, $F_{t}(M^m)$ will always be bounded by a likewise standard hypersphere $S_t:=F^S(\bbs^{m+p-1})$ of radius $R(t)$
for each $t\in[0,T)$. Therefore, the maximal time interval $[0,T)$ in which the flowing submanifolds $F_t$ keep regular must be within the lifetime $[0,T^S)$ of the hyperspheres $S_t$. Moreover, we have shown that $\lim\limits_{t\to \hat T_1} R^2=0$. Therefore, we can obtain that $T\leq T^S=\hat T_1$.

If $T<\hat T_1$, then we may directly use Proposition \ref{prop4.1} to know that $|F|^2$ is uniformly bounded in the case that $\max|F_0|^2<m$. It thus follows that the only possibility should be that $\lim\limits_{t\to T}\max|h|^2=+\infty$.

If $T=\hat T_1$, then we have
$$
0\leq\lim_{t\to \hat T_1}\max|F|^2\leq\lim_{t\to\hat T_1}R^2=0
$$
implying that $\lim\limits_{t\to\hat T_1}\max|F|^2=0$.

Case (2): If $\min|F_0|^2>m$, we can take $R_0=\min|F_0|$. Then $R^2_0>m$ and the solution $F^S:\bbs^{m+p-1}\times[0,T^S)\to\bbr^{m+p}$ of the ODE \eqref{R'} with the initial submanifold being an origin-centered standard hypersphere $S_0$ of radius $R_0$, defines a maximal regular solution $R^2\equiv R^2(t):=|F^S|^2$, $0\leq t<T^S$. Furthermore, by Corollary \ref{cor4.6}, we have $T^S=\hat T_2$ with $\hat T_2$ being given by \eqref{t1t2} and $\lim\limits_{t\to \hat T_2} R^2=+\infty$.

On the other hand, for any $R'^2_0\in\left(\frac{1}{2}(m+\min|F_0|^2),\min|F_0|^2\right)$ and any $\veps\in(0,\min|F_0|^2-R'^2_0)$, there exists the unique solution $R'^2=R'^2(t)$, $t\in[0,T_2)$, of \eqref{R'}, where $T_2$ is defined by \eqref{T'} with $R^2_0=R'^2_0$. Then Proposition \ref{prop4.8} makes it sure that $|F|^2-R'^2>\veps$ for all $p\in M^m$ and $t\in[0,T)\cap[0,T_2)$.
Furthermore, since the solution $R'^2=R'^2(t)$ of \eqref{R'} continuously depends on the initial value $R'^2_0$, we can take the limit of the inequality as $R'^2_0\to\min|F_0|^2\equiv R^2_0$ to conclude that $|F|^2-R^2\geq 0$ everywhere on $M^m\times\big([0,T)\cap[0,\hat T_2)\big)$, since $\veps\to 0$ and $T_2\to\hat T_2$, as $R'^2_0\to\min|F_0|^2$.

Similar to Case (1), we can show that $T\leq\hat T_{2}$ by using a similar argument.

If $T<\hat T_2$, then we may directly know from Theorem \ref{thm5.1} that either $\lim\limits_{t\to T}\max|F|^2=+\infty$ or
$\lim\limits_{t\to T}\max|h|^2=+\infty$.

If $T=\hat T_2$, then
$$
\lim_{t\to\hat T_2}\min|F|^2\geq\lim_{t\to\hat T_2}R^2=+\infty,
$$
from which we obtain the conclusion that $\lim\limits_{t\to\hat T_2}\min|F|^2=+\infty$.
\end{proof}


\begin{thebibliography}{99}
\bibitem{a-l-r} L. J. Al\'ias, J. H. de Lira and M. Rigoli, Mean curvature flow solitons in the presence of conformal vector fields, J. Geom. Anal. 30 (2020), no. 2, 1466-1529. https://doi.org/10.1007/s12220-019-00186-3.
\bibitem{andr} B. Andrews, Contraction of convex hypersurfaces in Riemannian spaces, J. Differential Geom. 39 (1994), no. 2, 407-431. doi:10.4310/jdg/1214454878. https://projecteuclid.org/euclid.jdg/1214454878.
\bibitem{a-b} B. Andrews and C. Baker, Mean curvature flow of pinched submanifolds to spheres, J. Differential Geom. 85 (2010), no. 3, 357-395. doi:10.4310/jdg/1292940688. https://projecteuclid.org/euclid.jdg/1292940688.
\bibitem{b} C. Baker, The mean curvature flow of submanifolds of high codimension, arXiv:1104.4409v1 [math.DG].
\bibitem{b-n} C. Baker and H. T. Nguyen, Evolving pinched submanifolds of the sphere by mean curvature flow, arXiv:2004.12259v1 [math.DG].
\bibitem{b-r10} A. Borisenko and V. Miquel, Gaussian mean curvature flow, J. Evol. Equ. 10 (2010), no. 2, 413-423. https://doi.org/10.1007/s00028-010-0054-2.
\bibitem{b-r14} A. Borisenko and V. Rovenski, Gaussian mean curvature flow for submanifolds in space forms, In: V. Rovenski and P. Walczak (eds.), Geometry and its Applications, 39-50, Springer Proc. Math. Stat., 72, Springer, Cham, 2014. https://doi.org/10.1007/978-3-319-04675-4$\_$2.
\bibitem{bra} K. Brakke, The motion of a surface by its mean curvature, Mathematical Notes, 20. Princeton University Press, Princeton, N. J., 1978.
\bibitem{c-k} B. Chow and D. Knopf, The Ricci Flow: An Introduction, Mathematical Surveys and Monographs, 110. Amer. Math. Soc., Providence, RI, 2004. http://dx.doi.org/10.1090/surv/110.
\bibitem{dtu} D. M. DeTurck, Deforming metrics in the direction of their Ricci tensors, J. Differential Geom. 18 (1983), no. 1, 157-162. doi:10.4310/jdg/1214509286. https://projecteuclid.org/euclid.jdg/1214509286.
\bibitem{e} K. Ecker, Regularity theory for mean curvature flow, Progress in Nonlinear Differential Equations and their Applications, 57. Birkh\"auser Boston, Inc., Boston, MA, 2004.
\bibitem{g-l-w} S. Z. Guo, G. H. Li and C. X. Wu, Mean curvature flow with a forcing field in the direction of the position vector of submanifolds, Acta Math. Sinica(Chin. Ser.) 57 (2014), no. 1, 139-150.
\bibitem{ha82} R. S. Hamilton, Three-manifolds with positive Ricci curvature, J. Differential Geom. 17 (1982), no. 2, 255-306. doi:10.4310/jdg/1214436922. https://projecteuclid.org/euclid.jdg/1214436922.
\bibitem{hui84} G. Huisken, Flow by mean curvature of convex surfaces into spheres, J. Differential Geom. 20 (1984), no. 1, 237-266. doi:10.4310/jdg/1214438998. https://projecteuclid.org/euclid.jdg/1214438998.
\bibitem{hui86} G. Huisken, Contracting convex hypersurfaces in Riemannian manifolds by their mean curvature, Invent. Math. 84 (1986), no. 3, 463-480. https://doi.org/10.1007/BF01388742.
\bibitem{j-x03} H.-Y. Jian and X. W. Xu, The vortex dynamics of a Ginzburg-Landau system under pinning effect, Sci. China Ser. A 46 (2003), no. 4, 488-498.
\bibitem{j-l06} H.-Y. Jian and Y. N. Liu, Ginzburg-Landau vortex and mean curvature flow with external force field, Acta Math. Sin. (Engl. Ser.) 22 (2006), no. 6, 1831-1842. https://doi.org/10.1007/s10114-005-0698-y.
\bibitem{l-s} G. H. Li and I. Salavessa, Forced convex mean curvature flow in Euclidean spaces, Manuscripta Math. 126 (2008), no. 3, 333-351. https://doi.org/10.1007/s00229-008-0181-z.
\bibitem{l-x} L. Lei and H. W. Xu, New developments in mean curvature flow of arbitrary codimension inspired by Yau rigidity theory, Proceedings of the Seventh International Congress of Chinese Mathematicians, Vol. I, 327-348, Adv. Lect. Math. (ALM), 43, Int. Press, Somerville, MA, 2019(see also arXiv:2004.13607v1 [math.DG]).
\bibitem{l-z} X. X. Li and D. Zhang, The blow-up of the conformal mean curvature flow, Sci. China Math. 63 (2020), no. 4, 733-754. https://doi.org/10.1007/s11425-018-9367-5.
\bibitem{l-j07} Y. N. Liu and H.-Y. Jian, Evolution of hypersurfaces by the mean curvature minus an external force field, Sci. China Ser. A 50 (2007), no. 2, 231-239. https://doi.org/10.1007/s11425-007-2077-x.
\bibitem{l-j08} H.-Y. Jian and Y. N. Liu , Long-time existence of mean curvature flow with external force fields, Pacific J. Math. 234 (2008), no.2, 311-324. doi:10.2140/pjm.2008.234.311.
\bibitem{l-x-z12} K. F. Liu, H. W. Xu and E. T. Zhao, Mean curvature flow of higher codimension in Riemannian manifolds, arXiv.org/math.DG/1204.0107v1.
\bibitem{l-x-y-z11} K. F. Liu, H. W. Xu, F. Ye and E. T. Zhao, Mean curvature flow of higher codimension in hyperbolic spaces, Comm. Anal. Geom. 21 (2013), no. 3, 651-669. https://dx.doi.org/10.4310/CAG.2013.v21.n3.a8.
\bibitem{l-x-y-z} K. F. Liu, H. W. Xu, F. Ye and E. T. Zhao, The extension and convergence of mean curvature flow in higher codimension, Trans. Amer. Math. Soc. 370 (2018), no. 3, 2231-2262.  https://doi.org/10.1090/tran/7281.
\bibitem{mntl} S. Montiel, Unicity of constant mean curvature hypersurfaces in some Riemannian manifolds, Indiana Univ. Math. J. 48 (1999), no. 2, 711-748.
\bibitem{mull} W. W. Mullins, Two-dimensional motion of idealized grain boundaries, J. Appl. Phys. 27 (1956), no. 8, 900-904.  http://dx.doi.org/10.1063/1.1722511.
\bibitem{p-s-x} D. Pu, J. J. Su and H. W. Xu, Surfaces pinched by normal curvature for mean curvature flow in space forms, arXiv:2004.14258v1 [math.DG].
\bibitem{s-s} O. C. Schn\"urer and K. Smoczyk, Evolution of hypersurfaces in central force fields, J. Reine Angew. Math. 550 (2002), 77-95.
\bibitem{smo11} K. Smoczyk, Mean curvature flow in higher codimension: Introduction and survey, Global Differential Geometry 2012, 231-274.  https://doi.org/10.1007/978-3-642-22842-1$\_$9.
\bibitem{w} M.-T. Wang, Lectures on mean curvature flows in higher codimensions, Handbook of geometric analysis. No. 1, 525-543, Adv. Lect. Math. (ALM), 7, Int. Press, Somerville, MA, 2008 (see also arXiv:1104.3354v1[math.DG]).
\end{thebibliography}
\end{document}